\documentclass[11pt,a4paper]{article}
\title{\bf Wasserstein Convergence for Empirical Measures of Subordinated Fractional Brownian Motions on the Flat Torus}
\author{Huaiqian Li\footnote{Email: {\color{blue}huaiqianlee@gmail.com} 
}
\quad Bingyao Wu\footnote{Email: {\color{blue}bingyaowu@163.com}}
  \vspace{2mm}
\\
{\footnotesize Center for Applied Mathematics, Tianjin University, Tianjin 300072, China}
}

\date{}
\usepackage{amsfonts}
\usepackage{amsmath, amsthm, amsfonts, amssymb, color}
\usepackage{mathrsfs}
\usepackage{url}
\usepackage{xcolor}
\usepackage{exscale}
\usepackage{relsize}
\usepackage{appendix}
\usepackage{hyperref}

\newtheorem{thm}{Theorem}[section]
\newtheorem{cor}[thm]{Corollary}
\newtheorem{lem}[thm]{Lemma}
\newtheorem{prp}[thm]{Proposition}

\newtheorem{rem}[thm]{Remark}
\theoremstyle{definition}

\newcommand{\scr}[1]{\mathscr #1}
\definecolor{wco}{rgb}{0.5,0.2,0.3}

\numberwithin{equation}{section} \theoremstyle{remark}

\newcommand{\bes}{\begin{equation}\begin{split}}
\newcommand{\ees}{\end{split}\end{equation}}

\def\R{\mathbb R}  \def\ff{\frac}  
\def\p<{\preceq}
\def\N{\mathbb N}  

\def\dd{\delta}  \def\vv{\varepsilon} 
\def\<{\langle} \def\>{\rangle}  
  \def\nn{\nabla}  \def\E{\mathbb E}
\def\d{\text{\rm{d}}}  \def\aa{\alpha} 
  \def\si{\sigma} 
 \def\beq{\begin{equation}}  
 
\def\e{e}    
 \def\tt{\theta}
 \def\P{\mathbb P} 
\def\C{\scr C}           
   
\def\Z{\mathbb Z}  \def\ll{\lambda}
 
\def\i{{\rm in}}  
  
  \def\i{{\rm i}} 
\def\to{\rightarrow}
\def\W{\mathbb W}
 
 \def\i{{\rm i}}
\def\T{\mathbb T}

\def\vv{\varepsilon}

\theoremstyle{definition}

\begin{document}
\allowdisplaybreaks
\maketitle
\makeatletter 
\renewcommand\theequation{\thesection.\arabic{equation}}
\@addtoreset{equation}{section}
\makeatother 

\begin{abstract}
We estimate rates of convergence for empirical measures associated with the subordinated fractional Brownian motion to the uniform distribution on the flat torus under the Wasserstein distance $\W_p$ for all $p\geq1$. In particular, our results coincides with recent ones on the diffusion process and the fractional Brownian motion. As an application, we provide similar results for time-discretized subordinated fractional Brownian motions.
\end{abstract}



\section{Introduction and main results}
Let $\T^d=\R^d/\Z^d$ be the $d$-dimensional flat torus endowed with the distance
$$\rho(x,y)=\min_{k\in\Z^d}|x-y-k|,\quad x,y\in\T^d,$$
where $|x|=\sqrt{\<x,x\>}$ denotes the Euclidean norm of $x\in\R^d$ and $\<\cdot,\cdot\>$ denotes the inner product on $\R^d$. Let $\scr{P}$ be the class of Borel probability measures on $\T^d$.
Given $\mu,\nu\in\scr{P}$, for any $p\in(0,\infty)$, the Wasserstein (or Kantorovich) distance of order $p$ between $\mu$ and $\nu$ induced by $\rho$  is defined as
$$\W_p(\mu,\nu)=\inf_{\pi\in\C(\mu,\nu)}\Big(\int_{\T^d\times \T^d}\rho(x,y)^p\,\pi(\d x,\d y)\Big)^{\min\{1,1/p\}},$$
where $\C(\mu,\nu)$ is the set of all probability measure on the product space $\T^d\times \T^d$ with marginal distributions $\mu$ and $\nu$, respectively. There are many literatures on the study of the Wasserstein distance, especially its connections with the optimal transport theory; see e.g. \cite[Chapter 5]{ChenMF2004} and \cite{Villani2008} for more details.

Let $\mathfrak{m}$ be the uniform distribution on the torus $\T^d$. The main purpose of this paper is to investigate the rate of convergence of Wasserstein distances between empirical measures associated with the subordinated fractional Brownian motion and $\mathfrak{m}$. In order to present our main results, we should introduce some basics on the subordinated fractional Brownian motion. For instance, one may refer to \cite{DS,LS} for further studies on subordinated fractional Brownian motions.

We begin with the definition of the fractional Brownian motion (abbr fBM). Let $H\in(0,1)$. We use $X^H:=(X_t^H)_{t\ge 0}$ to denote the fBM on $\R^d$ with Hurst index $H$, i.e., $X^H$ is a centered continuous-time Gaussian process with covariance matrix
$$\ff 1 2(s^{2H}+t^{2H}-|t-s|^{2H})\,{\rm I}_d,\quad s,t\ge 0,$$
where ${\rm I}_d$ is the $d\times d$ identity matrix. As we know, $X^H$ is neither a Markov process nor a semimartingale in general. However, in particular, if $H=1/2$, then $X^H$ is indeed the standard Brownian motion on $\R^d$. A fBM on the $d$-torus $\T^d$ is the natural projection of a fBM on $\R^d$. See e.g. \cite[Chapter 5]{Nualart} and \cite{N} for more properties on the fBM.

Now we recall some basics on the Bernstein function and the subordinator. A function $B$ from $C([0,\infty); [0,\infty))\cap C^\infty((0,\infty);[0,\infty))$ is called a Bernstein function if, for each $ k\in\mathbb N$,
 $$ (-1)^{k-1} \ff{\d^k}{\d \ll^k}B(\ll)\ge0, \quad \ll>0.$$
It is well known that every Bernstein function $B$  with $B(0)=0$ is characterized by the unique L\'{e}vy--Khintchine representation
\begin{equation}\label{LK}
B(\ll)=b\ll+\int_0^\infty(1-\e^{-\ll y})\,\nu(\d y),\quad \lambda\geq0,
\end{equation}
for some constant $b\ge 0$ and a L\'{e}vy measure $\nu$ on $(0,\infty)$ (i.e., a Radon measure on the Borel $\sigma$-algebra of $(0,\infty)$ such that $\int_0^\infty \frac{x}{1+x}\,\nu(\d x)<\infty$).
We need the following class of Bernstein functions (see e.g. \cite{WangWu}), i.e.,
$$ {\bf B}:= \big\{B:\ B\text{\ is\ a\ Bernstein\ function\ with } B(0)=0,\, B'(0)>0\big\}.$$
Let $B\in {\bf  B}$.  It is also well known that there exists a unique subordinator corresponding to $B$, denoted by $S^B=(S_t^B)_{t\geq0}$, i.e., an increasing stochastic process with stationary, independent increments, taking values in $[0, \infty)$ and $S_0^B=0$ such that $B$ is the Laplace exponent of $S^B$ given by
\begin{equation}\label{LT} \E \e^{-\ll S_t^B}= \e^{-t B(\ll)},\ \ t,\ll \ge 0.\end{equation}
For instance, the particular stable subordinator with index $\aa\in(0,1)$ is the process $S^B$ corresponding to $B(\ll)=\ll^\aa$, which has the representation \eqref{LK} with $b=0$ and
$$\nu(\d y)=\frac{-1}{\Gamma(-\alpha)}y^{-1-\alpha}\,\d y,$$
where $\Gamma(\cdot)$ stands for the Gamma function and $-\alpha\Gamma(-\alpha)=\Gamma(1-\alpha)$. We use $\P(S_t^B\in\cdot)$ to denote the distribution of $S_t^B$ in the sequel.  The following subclasses of $\mathbf{B}$ are also  needed; see e.g. \cite{LiWu}. For any $\aa\in[0,1]$, let
$$\textbf{B}^\aa=\Big\{B\in\textbf{B}:\ \liminf_{\ll\to\infty}\ll^{-\aa}B(\ll)>0\Big\},\quad
\textbf{B}_\aa=\Big\{B\in\textbf{B}:\ \limsup_{\ll\to\infty}\ll^{-\aa}B(\ll)<\infty\Big\}.$$
For every $\aa\in(0,1]$, the typical example  $B(\ll)=\ll^\aa$ belongs to the intersection of $\textbf{B}^\aa$ and $\textbf{B}_\aa$. For many other interesting examples, one may refer to the tables in \cite[Chapter 16]{SSV2012} for instance. Moreover, it is easy to verify that  there exists a constant $c>0$ such that, for any $B\in\textbf{B}^\aa$ (resp. $B\in\textbf{B}_\aa$) with  $\aa\in[0,1]$,
\begin{equation}\label{bound-B}
B(\ll)\ge c\min\{\ll^\aa, \ll\}~(\mbox{resp. }B(\ll)\le c\ll^\aa),\quad\ll\ge 0;
\end{equation}
see e.g. (4.11) and (5.8) in \cite{LiWu}.

The subordinated fractional Brownian motion (abbr sfBM) is defined by the time-change of the fBM as follows. Let $B\in\textbf{B}$, $H\in(0,1)$ and $S^B$ be the subordinator corresponding to $B$ such that $S^B$ and the fBM $X^H$ are independent. Set
$$X_t^{B,H}:=X_{S_t^B}^H,\quad t\ge 0.$$
The process $X^{B,H}:=(X_t^{B,H})_{t\ge 0}$ is called subordinated fractional Brownian motion with index $H$ corresponding to the Bernstein function $B$. In particular, when $H=\ff 1 2$, $X^B:=X^{B,H}$ is the subordinated Brownian motion (abbr sBM) corresponding to $B$, and when $B$ is the identity map, $X^{B,H}$ coincides with $X^H$. Refer to \cite{BSW,SSV2012,Bertoin97} for more details on the subordinated process and the Bernstein function.

We are concerned with empirical measures associated with the sfBM $X^{B,H}$, which are denoted as
$$\mu_t^{B,H}=\ff 1 t \int_0^t \dd_{X_s^{B,H}}\,\d s,\quad t>0.$$
In particular, when $H=1/2$, we write $\mu_t^B$ instead of $\mu_t^{B,1/2}$ for short.

To state our main results, further notations are needed. We write $a\lesssim b$ or $b\gtrsim a$ if there exists a positive constant $c$ such that $a\leq c b$, where $c$ may depending on the parameters $p,d$ and the Bernstein function $B$. We write $a\asymp b$ if $a\lesssim b$ and $b\lesssim a$ hold simultaneously.

Now we are ready to introduce our main results. First, for general $H\in (0,1)$, we present the following upper bound on the rate of convergence for empirical measures associated with the sfBM to $\mathfrak{m}$ under the Wasserstein distance.
\begin{thm}[Upper bound estimates]\label{TH2}
Assume that
\begin{itemize}
\item[\textnormal{(i)}] $H=1/2$ and $B\in\mathbf{B}^\aa$ for some $\aa\in[0,1]$,
\item[\textnormal{(ii)}] $H\neq 1/2$, and $B$ is given by \eqref{LK} satisfying that $\nu(\d y)\ge c y^{-1-\aa}\,\d y$ for some constants $c>0$ and $\aa\in(0,1)$.
\end{itemize}
Let $X^{B,H}$ be a $\T^d$-valued sfBM with index $H$ corresponding to $B$.  Then for any $p\geq1$ and any large enough $t>0$,
\begin{equation}\begin{split}\label{W1-upper}
\E\big[\W_p\big(\mu_t^{B,H},\mathfrak{m}\big)\big]\lesssim
\begin{cases}
t^{-1/2},\quad & d<2+ \aa/ H,\\
\big(\ff{\log t} t\big)^{1/ 2},\quad & d=2+\aa/ H,\\
t^{-\ff 1 {d-\aa/H}},\quad & d>2+\aa/ H.\\
\end{cases}
\end{split}\end{equation}
\end{thm}

\begin{rem}
Let $\alpha\in(0,1)$ and $B$ be a Bernstein function given by \eqref{LK}. If $\nu(\d y)\ge c y^{-1-\aa}\,\d y$ for some
constant $c>0$, then $B\in\mathbf{B}^\aa$. Indeed, on the one hand, by the dominated convergence theorem,
$$B'(\ll)=b+\int_0^\infty\e^{-\ll y}y\,\nu(\d y)\geq b+ c\int_0^\infty y^{-\aa}\e^{-\ll y}\,\d y,\quad\lambda>0,$$
which clearly implies that $B'(0)=\lim_{\lambda\rightarrow0^+}B'(\ll)>0$, and on the other hand,
$$\liminf_{\ll\rightarrow\infty}\frac{B(\ll)}{\ll^\aa}\geq \liminf_{\ll\rightarrow\infty}\Big[b\ll^{1-\aa}+c\int_0^\infty (1-\e^{-t})t^{-1-\aa}\,\d t \Big]>0.$$

However, the converse seems not true. Here is a counterexample. Letting $\alpha\in(0,1)$ and taking $\tilde{B}(\ll)=(1+\ll)^\aa-1$ for every $\lambda\geq0$, we can easily check that $\tilde{B}\in\mathbf{B}^\aa$, and the L\'{e}vy measure corresponding to $\tilde{B}$ is $\tilde{\nu}(\d y)=\frac{\aa}{\Gamma(1-\aa)}\e^{-y}y^{-1-\aa}\,\d y$. But it is clear that we are impossible to find any constants $C>0$ and $\beta\in(0,1)$ such that $\tilde{\nu}(\d y)\geq Cy^{-1-\beta}\,\d y$. For other counterexamples, refer to \cite[Chapter 16]{SSV2012} for instance.
\end{rem}

Next, we give the lower bound estimates on   $\W_p$ between empirical measures $\mu_t^{B,H}$ associated with the sfBM and $\mathfrak{m}$ as follows.
\begin{thm}[Lower bound estimates]\label{TH1}
Let $H\in(0,1)$ and $B\in\mathbf{B}$. Let $X^{B,H}$ be a $\T^d$-valued sfBM with index $H$ corresponding to $B$. 
\begin{itemize}
 \item[\textnormal{(1)}] Assume that $H= 1/2$,  $B\in\mathbf{B}_\aa\cap\mathbf{B}^\aa$ for some $\aa\in[0,1]$. Then for any
 $p\geq1$ and any large enough $t>0$,
\begin{equation}\label{1TH1}
\E\left[\W_p\left(\mu_t^B,\mathfrak{m}\right)\right]\gtrsim
\begin{cases}
t^{-1 /2},\quad & d<2(1+\aa),\\
\big(\ff {\log t} t\big)^{ 1/ 2},\quad & d=2(1+\aa),\\
t^{- \ff 1 {d-2\aa}},\quad & d>2(1+\aa).
\end{cases}
\end{equation}
\item[\textnormal{(2)}] Assume that $B\in\mathbf{B}_\aa$ for some $\aa\in(0,1]$ and $d>\aa/H$. Then for any $p>0$ and any large enough $t>0$,
\begin{equation}\label{THL1}
\E\big[\W_p\big(\mu_t^{B,H},\mathfrak{m}\big)\big]\gtrsim t^{-\frac{\min\{1,p\}}{d-\aa/H}}.
\end{equation}
\end{itemize}
\end{thm}

\begin{rem}
\textnormal{(1)} When $H=1/2$ and $p=2$,  the above rate of convergence coincides with the one in \cite{WangWu} obtained by quite a different approach for subordinated diffusion processes on compact Riemannian manifolds. Moreover, we should emphasize that the assumption on the Bernstein function $B$ is weaker than the one in \cite{WangWu} since we do not require additionally that $B\in\mathbb{B}$, where
$$\mathbb{B}:= \Big\{B\in {\bf B}:\  \int_1^\infty s^{\ff d 2 -1}\e^{-r B(s)}\,\d s<\infty\mbox{ for all }r>0\Big\}.$$
For more details on the relation between $\mathbb{B}$ and $\mathbf{B}^\aa,\,\mathbf{B}_\aa$, one may refer to \cite[Remark 1.1]{LiWu}, where extensions of \cite{WangWu} to complete (not necessarily compact) Riemannian manifolds possibly with boundary are also obtained.

\textnormal{(2)} On the one hand, as a straightforward consequence of Theorems \ref{TH2} and \ref{TH1}, when $H=1 / 2$ and $B\in\mathbf{B}_\aa\cap\mathbf{B}^\aa$ for some $\aa\in[0,1]$, we establish both the upper and the comparable lower bounds on $\W_p$ for all $p\ge 1$ and all dimensions,
\begin{equation*}
\E\left[\W_p\left(\mu_t^B,\mathfrak{m}\right)\right]\asymp
\begin{cases}
t^{-1 /2},\quad & d<2(1+\aa),\\
\big(\ff {\log t} t\big)^{1/ 2},\quad & d=2(1+\aa),\\
t^{- \ff 1 {d-2\aa}},\quad & d>2(1+\aa).
\end{cases}
\end{equation*}
On the other hand, for general $H\in(0,1)$, if $d>2+\aa/H$ and $B$ is given by \eqref{LT} satisfying $\nu(\d y)\ge c y^{-1-\aa}\d y$ for some constants $c>0$ and $\aa\in(0,1)$, we also obtain the following precise convergence rate
$$\E\big[\W_p\big(\mu_t^{B,H},\mathfrak{m}\big)\big]\asymp t^{-\ff 1 {d-\aa/H}},\quad p\ge 1.$$
\end{rem}

As an application,  in the next theorem, we give a discrete time approximation version of Theorems \ref{TH1} and \ref{TH2}, which may be interesting in numerical simulations of sfBM for instance; see e.g. the very recent paper \cite{DJL} on numerical simulations of mean-field Ornstein--Uhlenbeck process. For every nonnegative number $a$, let $\lfloor a\rfloor$ denote the greatest integer less than or equal to $a$.
\begin{thm}\label{W1TU}
Suppose that
\begin{itemize}
\item[\textnormal{(1)}] $H= 1/2$ and $B\in\mathbf{B}^\aa$ for some $\aa\in[0,1]$,
\item[\textnormal{(2)}] $H\in(0, 1/2)\cup ( 1/2 ,1)$,
    and $B$ is a Bernstein function represented by \eqref{LK} such that $\nu(\d y)\ge c y^{-1-\aa}\,\d y$ for some constants $c>0$ and $\aa\in(0,1)$.
\end{itemize}
Let $X^{B,H}$ be a $\T^d$-valued sfBM with index $H$ corresponding to $B$. Let $\beta>0$ and set $\tau\asymp t^{-\beta}$ for any $t>0$. Then for any large enough $t>0$,
\begin{equation*}
\E\big[\W_2\big(\mu_{\tau,t}^{B,H},\mathfrak{m}\big)\big]\lesssim
\begin{cases}
t^{- 1 /2},\quad &d\le 2,\\
t^{-\min\{ 1 /2, (1+\beta)/d\}},\quad & 2<d<2+\aa/H,\\
\max\Big\{\sqrt{\ff{\log t}{t}}, t^{-\ff {1+\beta}d}\Big\},\quad & d=2+\aa/H,\\
t^{-\min\left\{d-\ff \aa H,\ff {1+\beta}d\right\}},\quad &d>2+\aa/H,
\end{cases}
\end{equation*}
where
$$
\mu_{\tau,t}^{B,H}:=\ff 1{\lfloor t/\tau\rfloor}\sum_{k=1}^{\lfloor t/\tau\rfloor}\dd_{X_{k\tau}^{B,H}},\quad t\geq\tau>0.
$$
\end{thm}

The direct motivation for the above study is two fold. On the one hand, recently, rates of convergence and even exact limits of empirical measures associated with subordinated diffusion processes on compact and noncompact Riemannian manifolds under $\W_2$ are investigated, where the Markov property plays a crucial role; see \cite{WangWu,LiWu,LiWu1,LiWu2}. However, in the setting of the aforementioned papers,  the related questions on general Wasserstein distance $\W_p$, especially for $p> 2$, are still open. On the other hand, very recently, rates of convergence of empirical measures associated with fBMs were obtained in \cite{HMT}. So it should be interesting to study the rates of convergence of empirical measures associated with sfBMs under the Wasserstein distance.

In the literature, the study on asymptotic behaviours of Wasserstein distances between empirical measures associated with i.i.d. random variables and the reference measure, particularly on estimating the rate of convergence, has received lots of attentions; see e.g.
\cite{FG2015,Tal2014,DSS2013,GL2000,AKT1984}. In the breakthrough paper \cite{eAST} (where the precise limit is proved), a new PDE method was introduced, which inspired many recent studies on the convergence of empirical measures under Wasserstein distances; refer to \cite{Zhu,Bor2021,BoLe2021,LedZhu2021,St2021a,St2021b,BL2019,Led2017} for the case of i.i.d. (including weakly dependent) random variables  and \cite{eW2,eW3,eW1,eWZ} for the case of diffusion processes, as well as \cite{eW4} for the case of stochastic partial differential equations.

The rest of this paper is organized as follows.  In Section 2, we  briefly introduce some basics on the Fourier analysis on $\T^d$ and recall some known results which will be used to prove our main results. Sections 3, 4  and 5 are devoted to prove Theorems \ref{TH2},  \ref{TH1} and \ref{W1TU}, respectively. We should mention that the proof is motivated by the aforementioned recent works \cite{HMT} and \cite{WangWu}. For the reader's convenience, an appendix is included to provide some long but elementary calculations as a supplement to the proof of Lemma \ref{SDU}.

\section{Preparations}
In this part, we first recall necessary basic facts on Fourier analysis on the torus $\T^d$ and then introduce some known results which will be used in the sequel. We may identify $\T^d$ with the cube $[-\frac{1}{2}, \frac{1}{2}]^d$ in $\R^d$, and we identify the measure on $\T^d$ with the restriction of the Lebesgue measure $\d x$ on $[-\frac{1}{2}, \frac{1}{2}]^d$. Functions on $\T^d$ may be thought as functions $f$ defined on $\R^d$ such that $f$ is $1$-periodic in each variable, i.e., $f(x+\xi)=f(x)$ for all $x\in\T^d$ and $\xi\in\mathbb{Z}^d$.  For a detailed study on Fourier analysis on the flat torus, see e.g. \cite{DymMc}. Let $\mathbb{N}=\{1,2,\cdots\}$.

For every $p\in[1,\infty]$, denote the classic $L^p$ (resp. $l^p$) space over $\T^d$ (resp. $\mathbb{Z}^d$) by $L^p(\T^d)$  (resp. $l^p(\mathbb{Z}^d)$) with norm $\|\cdot\|_{L^p(\T^d)}$ (resp. $\|\cdot\|_{l^p(\mathbb{Z}^d)}$). Let $\scr{M}$ be the class of all finite signed Borel measures on $\T^d$. We denote
 the imaginary unit by $\i$.

For any $\mu\in\scr{M}$ and any $f\in L^1(\T^d)$, we use
$$\hat{\mu}(\xi)=\int_{\mathbb{T}^d}\exp(-2\pi{\rm i}\langle\xi, x\rangle)\,\mu(\d x),\quad \hat{f}(\xi)=\int_{\mathbb{T}^d}\exp(-2\pi {\rm i}\<\xi, x\>)f(x)\,\d x, \quad \xi\in \Z^d,$$
to denote the Fourier transforms of $\mu$ and $f$, respectively. For a vector valued function, we define its Fourier transform as the Fourier transform of all components. Then, by the inverse Fourier transform, for a sufficiently smooth function $f$ on $\T^d$, we can express $f$ as an absolutely convergent Fourier series, i.e.,
$$f(x)=\sum_{\xi\in\mathbb{Z}^d}\exp(2\pi {\rm i}\<\xi, x\>)\hat{f}(\xi),\quad x\in\T^d.$$
According to this and the orthogonality, i.e.,
$$\int_{\mathbb{T}^d}\exp(2\pi{\rm i}\<\xi,x\>)\d x=\dd_0(\xi),\quad \xi\in\mathbb{Z}^d,$$
one can verify that (see \cite[(2.2)]{HMT}), for any $f\in L^{2n}(\mathbb{T}^d)$ and any $n\in\mathbb{N}$,
\begin{equation}\label{FP}
\int_{\mathbb{T}^d}|f|^{2n}(x)\d x=\sum_{\xi_1,\xi_2,\cdots,\xi_{2n}\in\Z^d}\prod_{i=1}^{2n}\hat{f}(\xi_i)\dd_0\Big(\sum_{i=1}^{2n}\xi_i\Big),
\end{equation}
where $\delta_0$ denotes the Dirac measure at the origin. In particular, when $n=1$, \eqref{FP} reduces to Parseval's identity
\begin{equation}\label{F2}
\|f\|_{L^2(\T^d)}=\|\hat{f}\|_{l^2(\mathbb{Z}^d)},\quad f\in L^2(\mathbb{T}^d).
\end{equation}
Combining \eqref{F2} and the simple inequality $\|f\|_{L^\infty(\T^d)}\leq\|\hat{f}\|_{l^1(\mathbb{Z}^d)}$, by the Riesz--Thorin interpolation theorem (see e.g. \cite[page 3]{Davies89}), we immediately derive the following Hausdorff--Young inequality, i.e.,
\begin{equation}\label{HY}
\Big(\int_{\T^d}|f|^p(x)\d x\Big)^{1/p}\le\Big(\sum_{\xi\in\Z^d}|\hat{f}|^q(\xi)\Big)^{1/q},\quad f\in L^p(\T^d),
\end{equation}
whenever $p\in[2,\infty]$ and $ q:=\ff p {p-1}$.

Let $(P_t)_{t\ge 0}$ be the heat semigroup/flow corresponding to the standard Brownian motion on $\T^d$. Then for each $\mu\in\scr{M}$, $P_t\mu$ can be expressed as the convolution of $\mu$ and the Gaussian kernel $q_t$ on $\T^d$, i.e.,
$$P_t\mu(x)=\int_{\T^d}q_t(x-y)\,\mu(\d y),\quad x\in\T^d,\,t>0,$$
where
$$q_t(x)=\frac{1}{(2\pi t)^{d/2}}\sum_{k\in\mathbb{Z}^d}\exp\Big(-\frac{|x-k|^2}{2t}\Big),\quad t>0,\,x\in\T^d.$$
By the Fourier transform, we have
\begin{equation}\label{FM}
\widehat{P_t \mu}(\xi)=\exp(-2\pi^2 t|\xi|^2)\hat{\mu}(\xi),\quad \xi\in\Z^d,\,t>0.
\end{equation}

Let $\mu\in\scr{M}$ such that $\mu(\T^d)=0$. Assume that $h$ is a solution to the Poisson's equation on $\T^d$, i.e.,
\begin{equation*}\begin{split}
-\Delta u=\mu.
\end{split}\end{equation*}
It is easy to see that the Fourier transforms of $h$, $\nabla h$ and $\mu$ are closely related as follows, i.e., for every $\xi\in\mathbb{Z}^d$,
$$4\pi^2|\xi|^2\hat{h}(\xi)=\hat{\mu}(\xi),$$
and hence
\begin{equation}\label{GRA}
2\pi|\xi|^2\widehat{\nn h}(\xi)=\i\xi\hat{\mu}(\xi).
\end{equation}

For every $\epsilon>0$, define
$$\phi_\epsilon(\xi)=\ff{\exp(-\epsilon|\xi|^2)}{|\xi|+1},\quad \xi\in\mathbb{Z}^d.$$
Employing polar coordinate, by a careful computation, we have the following estimate
(see \cite[Lemma 3.3]{HMT} for a detailed proof).
\begin{lem}\label{g}
For any $d\ge 1$, $\epsilon>0$ and $p\ge 1$,
$$\|\phi_\epsilon\|_{l^p(\mathbb{Z}^d)}\asymp
\begin{cases}
1,\quad &d<p,\\
|\log\epsilon|^{1/p},\quad &d=p,\\
\epsilon^{-\ff 1 2(d/p-1)},\quad &d>p.
\end{cases}$$
\end{lem}

We now borrow a lemma from \cite[Lemma 2.1]{HMT}, one can also refer to \cite[Proposition 5.3]{eWZ} and  \cite[Appendix]{eW3}, which contains upper and lower bounds on the Wasserstein distance $\W_p$  given by $L^p$ norms of the gradient of solutions to the Poisson's equation regularized by the heat flow $(P_t)_{t>0}$ on $\T^d$.
Let $H^1(\T^d)$ be the Sobolev space defined as
$$H^1(\T^d)=\Big\{u\in L^2(\T^d):\ \sum_{k\in\mathbb{Z}^d}\big(1+|k|^2\big)|\hat{u}(k)|^2<\infty \Big\}.$$

\begin{lem}\label{W}
Let $\mu,\nu\in\scr{P}$, and for any $\vv>0$, let $u_\vv\in H^1(\mathbb{T}^d)$ be the solution to the following Poisson's equation on $\T^d$, i.e.,
$$-\Delta u= P_{\vv}(\mu-\nu).$$
Then, there exists a constant $c>0$ such that
\begin{equation}\label{W1}\W_1(\mu,\nu)\le c\inf_{\vv>0}\big\{\vv^{1/ 2}+\|\nn u_\vv\|_{L^2(\T^d)}\big\},\end{equation}
and
\begin{equation}\label{W1L}
\W_1(\mu,\nu)\ge \sup_{\kappa,\vv>0}\Big\{\ff 1 \kappa\|\nn u_\vv\|_{L^2(\T^d)}^2-\ff c {\kappa^3}\|\nn u_\vv\|_{L^4(\T^d)}^4\Big\}.\end{equation}
In addition, if $\nu=\mathfrak{m}$, then for any $p>1$, there exists a constant $C>0$ such that
\begin{equation}\label{WPU}
\W_p^p(\mu,\mathfrak{m})\le C\inf_{\vv>0}\Big\{\vv^{p/2}+\|\nn u_\vv\|_{L^p(\T^d)}^p\Big\}.
\end{equation}
\end{lem}

In the sequel, for each $m\in\mathbb{N}$, let $\mathcal{S}_m$ denote the group of all permutations of the set $\{1,2,\cdots,m\}$.

\section{Upper bounds}
With \eqref{WPU} in hand, in order to prove the upper bound in Theorem \ref{TH2}, it suffices to estimate the upper bound of $\|\nabla u_\vv\|_{L^p(\T^d)}$, where for every $\vv>0$, $u_\vv$ is the solution to the following regularized Poisson's equation, i.e.,
\begin{equation}\label{PoisE-H}
-\Delta u=P_\varepsilon(\mu_t^{B,H}-\mathfrak{m}).
\end{equation}
To this end, we establish first the following moment estimate for the subordinator,  which may be interesting in its own right.
\begin{lem}\label{SDU}
Let 
$S^B$ be a subordinator with Bernstein function $B$ given by \eqref{LK}.  Assume that $\nu(\d y)\ge c y^{-1-\aa}\,\d y$ for some constants $c>0$ and $\aa\in(0,1)$. Then for any $\dd>0$, there exists some constant $c_1>0$ depending only on $c,\alpha,\delta$ such that, for any $t>0$ and any $\ll>0$,
\begin{align*}\E[\e^{-\ll (S_t^B)^\dd}]\leq
\begin{cases}
e\cdot\exp\Big(-c_1\ll^{\ff {\aa}{(1-\dd)\aa+\dd}}t^{\ff{\dd}{(1-\dd)\aa+\dd}}\Big),\quad &
\dd\in(0, 1],\\
e\cdot\exp\Big(-c_1\ll^{\ff {\aa}{(1-\dd)\aa+\dd^2}}t^{\ff{\dd}{(1-\dd)\aa+\dd^2}}\Big),\quad &\dd\in
(1,\infty).
\end{cases}
\end{align*}
\end{lem}
\begin{rem}
The explicit dependence of the parameter $\ll$  is crucial for our purpose; see the proof of Theorem \ref{TH2} below. As for the case when $\delta\in(0,1]$, the above estimate improves the one in \cite[Theorem 2.1]{DSS}, the latter of which was applied effectively to study ergodic properties of subordinated Markov processes (see \cite[Theorem 1.1]{DSS}). The estimate is sharp in the sense that, when $\dd=1$, the powers of $\ll$ and $t$ coincides with \eqref{LT} by taking $B(\ll)=\ll^\aa, \aa\in(0,1)$.
\end{rem}
\begin{proof}[Proof of Lemma \ref{SDU}] We give a detailed proof for the case when $\delta>1$. The proof of the case when $\delta\in(0,1]$ is similar; see also \cite[pages 168--170]{DSS}.

(i) Without loss of generality, we may assume that the subordinator $(S_t^B)_{t\geq0}$ has no drift part, i.e., the infinitesimal generator  of $(S_t^B)_{t\geq0}$, denoted by $L$, is given by
$$L g(x)=\int_0^\infty[g(x+y)-g(x)]\,\nu(\d y),\quad g\in C_b^1(\R),\,x\in\R.$$
Let $\dd>1$ and $\ll>0$. Set
$$g(x):=\e^{-\ll x^\dd},\quad x\ge 0.$$
Since $\nu(\d y)\ge cy^{-1-\aa}\,\d y$ and $\alpha\in(0,1)$, by the elementary inequalities
$$1-\e^{-r}\ge\ff{\e-1}{\e}\min\{1, r\},\quad (1+r)^\dd\ge 1+
r^\delta,\quad\quad r\ge0,\, \dd>1,$$
we have
\begin{equation}\begin{split}\label{LU}
Lg(x)&=\int_0^\infty \Big(\e^{-\ll(x+y)^\dd}-\e^{-\ll x^\dd}\Big)\,\nu(\d y)\cr
&\le cx^{-\aa}\e^{-\ll x^\dd}\int_0^\infty\Big(\e^{-\ll x^\dd[(1+z)^\dd-1]}-1\Big)\,\ff{\d z}{z^{1+\aa}}\cr
&\le\ff{c(1-\e)}{\e}x^{-\aa}\e^{-\ll x^\dd}\int_0^{(1+\ll^{-1}x^{-\dd})^{1/\dd}-1}\ff{\ll x^{\dd}[(1+z)^\dd-1]}{z^{1+\aa}}\,\d z\cr
&\le\ff{c\ll(1-\e)}{\e}x^{\dd-\aa}\e^{-\ll x^\dd}\int_0^{(1+\ll^{-1}x^{-\dd})^{1/\dd}-1}z^{\dd-\aa-1}\,\d z\cr
&=-C_1\psi(g(x)),
\end{split}\end{equation}
where $C_1:=\ff{c\ll^{\aa/\dd}(\e-1)}{\e(\dd-\aa)}$ and
$$\psi(u):=u\left[(1-\log u)^{1/\dd}-(-\log u)^{1/\dd}\right]^{\dd-\aa},\quad 0<u\le 1.$$

According to Dynkin's formula and \eqref{LU}, we have for every $s\in [0,t]$,
\begin{equation*}\begin{split}
\E[g(S_t^B)]-\E[g(S_s^B)]
\le-C_1\E\Big\{\int_s^t\psi\big((g(S_u^B)\big)\,\d u\Big\}\le-C_1\int_s^t\psi\big(\E g(S_u^B)\big)\,\d u,
\end{split}\end{equation*}
where the last inequality is due to the fact that $\psi$ is convex on $(0,1]$ (see Lemma \ref{C} in Appendix). Let
$$h(t):=\E g(S_t^B),\quad t\ge 0.$$
Then
$$\ff{h(t)-h(s)}{t-s}\le -\ff{C_1}{t-s}\int_s^t \psi(h(u))\,\d u,\quad 0\le s\le t.$$
Noticing that $h$ is absolutely continuous on $[0,\infty)$, we derive from the last inequality that
$$h'(s)\le-C_1\psi(h(s)),\quad \mbox{a.e. }s\ge 0.$$
Applying \cite[Lemma 2.1]{DSS} or \cite[Lemma 5]{SW}, by the fact that $\psi$ is increasing on $(0,1]$ (see also Lemma \ref{C}), we have
$$G(h(t))\le G(1)-C_1 t,\quad t\ge 0,$$
where
$$G(v):=-\int_v^1\ff{\d u}{\psi(u)},\quad 0< v\le 1.$$
It is easy to see that $G$ is strictly increasing with $\lim_{r\rightarrow0^+}G(r)=-\infty$ and $G(1)=0$. Indeed, by the mean value theorem,
$$(1-\log u)^{1/\dd}-(-\log u)^{1/\dd}\le\dd^{-1}(-\log u)^{1/\dd-1},\quad u\in(0,1],$$
and hence, by the change-of-variables formula,
\begin{equation*}\begin{split}
\int_v^1\ff{\d u}{\psi(u)}&\ge\dd^{\dd-\aa}\int_v^1\ff{\d u}{u(-\log u)^{(1/\dd-1)(\dd-\aa)}}
=\frac{\dd^{\dd-\aa}(-\log v)^{1-(1/\dd-1)(\dd-\aa)}}{1-(1/\dd-1)(\dd-\aa)}
\rightarrow\infty,\quad v\rightarrow0^+,
\end{split}\end{equation*}
where the last line is due to that $(1/\dd-1)(\dd-\aa)<0$ for any $\dd>1$. Thus, we arrive at
\begin{equation}\label{ii-1}
h(t)\le G^{-1}(G(1)-C_1t)=G^{-1}(-C_1t),\quad t\ge 0,
\end{equation}
where $G^{-1}$ denotes the inverse function of $G$.

Next, we give a lower bound on $G(v)$. By the change-of-variables formula,
\begin{equation*}\begin{split}
G(v)&=-\int_v^1 \frac{[(1-\log u)^{1/\dd}-(-\log u)^{1/\dd}]^{\aa-\dd}}{u}\,\d u\\
&=-\int_0^{-\log v}[(1+s)^{1/\dd}-s^{1/\dd}]^{\aa-\dd}\,\d s,\quad 0<v\le 1.
\end{split}\end{equation*}
Since $\dd>1$,  we have for any $s\ge 0$,
$$(1+s)^{1/\dd}-s^{1/\dd}=\ff 1 {\dd}\int_s^{1+s} u^{(1-\dd)/\dd}\,\d u\ge\ff 1 {\dd}(1+s)^{(1-\dd)/\dd}.$$
This implies that, for any $v\in(0,1]$,
\begin{equation*}\begin{split}
\int_0^{-\log v}[(1+s)^{1/\dd}-s^{1/\dd}]^{\aa-\dd}\,\d s &\le \int_0^{-\log v}\Big[\ff 1 {\dd}(1+s)^{\ff{1-\dd}\dd}\Big]^{\aa-\dd}\,\d s\\
&=C_2\Big[(1-\log v)^{\ff{\aa-\dd\aa+\dd^2}{\dd}}-1\Big],
\end{split}\end{equation*}
where $C_2:=\ff{\dd^{1+\delta-\alpha}}{(1-\dd)\aa+\dd^2}>0$. Thus, for any $v\in(0,1]$,
\begin{equation}\label{G-lbd}
G(v)\ge -C_2\Big[(1-\log v)^{\ff{\aa-\dd\aa+\dd^2}{\dd}}-1\Big].
\end{equation}

Therefore, by \eqref{ii-1} and \eqref{G-lbd}, we have
\begin{equation*}\begin{split}
\E g(S_t^B)&\le\exp\Big[1-\Big(\ff{C_1}{C_2}t+1\Big)^{\ff{\dd}{(1-\dd)\aa+\dd^2}}\Big]
\le \e\cdot\exp\Big(-C_3\ll^{\ff{\aa}{(1-\dd)\aa+\dd^2}}t^{\ff{\dd}{(1-\dd)\aa+\dd^2}}\Big),
\end{split}\end{equation*}
for some constant $C_3>0$ depending only on $c,\alpha,\delta$.

(ii) 
Let $\dd\in(0,1]$ and $\ll>0$. Define
$$f(x)=\e^{-\ll x^\dd},\quad x\ge 0,$$
and
$$\varphi(u)=u\Big[(1-\log u)^{1/\dd}-(-\log u)^{1/\dd}\Big]^{-\aa},\quad 0<u\le 1.$$
Let
$$F(v)=-\int_v^1\ff{\d u}{\varphi(u)},\quad 0<v\le 1,$$
and let $F^{-1}$ denote the inverse function of $F$. It is clear that $F$ is strictly increasing on $(0,1]$, $\lim_{r\to 0^+}F(r)=-\infty$ and
$F(1)=0$. In fact, by the change-of-variables formula,
\begin{equation*}
\int_v^1\frac{\d u}{\varphi(u)}=\int_0^{-\log v}\Big[(1+r)^{1/\delta}-r^{1/\delta}\Big]^\aa\,\d r\geq-\log v\rightarrow+\infty,\quad v\rightarrow 0^+.
\end{equation*}
By the proof of \cite[Theorem 2.1]{DSS}, we have
\begin{equation}\label{1SDU}
\E[f(S_t^B)]\le F^{-1}(-C_4 t),\quad t\ge 0.
\end{equation}
where $C_4:=c(1-\e^{-1})\aa^{-1} \ll^{\aa/\dd}$.

Moreover, it is easy to see that
$$F(v)=-\int_0^{-\log v}\Big[(1+s)^{1/\delta}-s^{1/\delta}\Big]^\aa\,\d s,\quad v\in(0,1].$$
Since
$$
(1+s)^{1/\dd}-s^{1/\dd}=\ff 1 {\dd}\int_s^{1+s} u^{(1-\dd)/\dd}\,\d u\le\ff 1 {\dd}(1+s)^{(1-\dd)/\dd},\quad s\geq0,
$$
we have
\begin{equation*}\begin{split}
\int_0^{-\log v}\Big[(1+s)^{1/\dd}-s^{1/\dd}\Big]^\aa\d s&\le\ff 1 {\dd^\aa}\int_0^{-\log v}(1+s)^{\ff{(1-\dd)\aa}{\dd}}\,\d s\\
&=\ff{\dd^{1-\aa}}{(1-\dd)\aa+\dd}\Big[(1-\log v)^{\ff{(1-\dd)\aa+\dd}{\dd}}-1\Big],\quad v\in(0,1].
\end{split}\end{equation*}
Hence
\begin{equation}\label{2SDU}
F(v)\ge-\ff{\dd^{1-\aa}}{(1-\dd)\aa+\dd}\Big[(1-\log v)^{\ff{(1-\dd)\aa+\dd}{\dd}}-1\Big],\quad v\in(0,1].
\end{equation}
Letting
$$C_4 t=\ff{\dd^{1-\aa}}{(1-\dd)\aa+\dd}\Big[(1-\log v)^{\ff{(1-\dd)\aa+\dd}{\dd}}-1\Big],$$
we obtain that
$$v=\exp\Big[1-\Big(\ff{C_4[(1-\dd)\aa+\dd]}{\dd^{1-\aa}}t+1\Big)^{\ff{\dd}{(1-\dd)\aa+\dd}}\Big].$$
According to \eqref{2SDU} and the monotonicity of $F$, we arrive at
\begin{equation}\label{G-upper}
F^{-1}(-C_4t)\le\exp\Big[1-\Big(\ff{C_4[(1-\dd)\aa+\dd]}{\dd^{1-\aa}}t+1\Big)^{\ff{\dd}{(1-\dd)\aa+\dd}}\Big],\quad t>0.
\end{equation}

Combining \eqref{G-upper} and \eqref{1SDU} together, we find some constant $C_5>0$ such that
$$\E[f(S_t^B)]\le e\cdot\exp\left[-C_5\ll^{\ff {\aa}{(1-\dd)\aa+\dd}}t^{\ff{\dd}{(1-\dd)\aa+\dd}}\right],\quad t>0,$$
where $C_5$ depends only on $\aa,\dd,c$.
\end{proof}

The moment estimate for the subordinator plays an important role in proving the following lemma.
\begin{lem}\label{LTU}
Under the Assumptions in Theorem \ref{TH2}, for any $p\in \mathbb{N}$ and any $\xi_1,\cdots,\xi_p\in\mathbb{Z}^{d}$,
$$\Big|\E\Big[\prod_{j=1}^p\widehat{\mu_t^{B,H}}(\xi_j)\Big]\Big|\lesssim \ff 1 {t^p}\sum_{\si\in\mathcal{S}_p}\prod_{j=1}^p\min\Big\{\ff 1
{|\sum_{i=j}^p\xi_{\si_i}|^{\aa/H}},t\Big\},\quad t>0.$$
\end{lem}

\begin{proof} We divide the proof into two parts.

\underline{\textbf{Part 1}}. We first consider the case when $H=1/2$ and $B\in\textbf{B}^\aa$ for some $\aa\in[0,1]$.

Let $t>0$ and $p\in \mathbb{N}$.  Then for any $j=1,2,\cdots, p$, we have
$$\widehat{\mu_t^B}(\xi_j)=\ff 1 t \int_0^t\exp(-2\pi {\rm i}\<\xi_j,X_s^B\>)\,\d s,\quad \xi_j\in\mathbb{Z}^d.$$
Hence
\begin{equation}\begin{split}\label{1LTU}
\prod_{j=1}^p\widehat{\mu_t^B}(\xi_j)
&=\ff 1 {t^p}\prod_{j=1}^p\Big(\int_0^t\exp(-2\pi {\rm i}\<\xi_j,X_s^B\>)\,\d s\Big)\\
&=\ff 1 {t^p}\sum_{\sigma\in \mathcal{S}_p}\int_{\Delta_\sigma}\exp\Big(-2\pi {\rm i}\sum_{j=1}^p \<\xi_j,X_{t_j}^B\>\Big)\,\d t_1 \d t_2\cdots \d t_p,
\end{split}\end{equation}
where for each $\sigma\in\mathcal{S}_p$,
$$\Delta_\sigma:=\{(t_{\sigma_1},t_{\sigma_2},\cdots,t_{\sigma_p})\in\R^p:\ 0\le t_{\sigma_1}\le t_{\sigma_2}\le\cdots\le t_{\sigma_p}\le t\}.$$
In other words, we divide $[0,t]^p$ into $p!$ many distinct simplexes.

Let $t_{\sigma_0}=0$ and $(t_{\sigma_1},t_{\sigma_2},\cdots,t_{\sigma_p})\in\Delta_\sigma$.
Noting that $(X_t^B)_{t\ge 0}$ has independent increments and $X_t^B-X_s^B$ and $X^B_{t-s}$ have the same distribution for every $t\geq s\geq0$, we arrive at
\begin{equation*}\begin{split}
\E\Big[\exp(-2\pi {\rm i}\<\xi_j,X_{t_j}^B\>)|X_{t_{j-1}}^B\Big]&=\e^{-2\pi{\rm i}\<\xi_j,X_{t_{j-1}}^B\>}\E\Big[-2\pi{\rm i}\<\xi_j,X_{t_j}^B-X_{t_{j-1}}^B\>|X_{t_{j-1}}^B\Big]\\
&=\e^{-2\pi{\rm i}\<\xi_j,X_{t_{j-1}}^B\>}\E\Big[-2\pi{\rm i}\<\xi_j,X_{t_j-t_{j-1}}^B\>\Big],\quad j=1,2,\cdots,p.
\end{split}\end{equation*}
Then
\begin{equation}\begin{split}\label{01LTU}
\E\Big[\exp\Big(-2\pi {\rm i }\sum_{j=1}^p\<\xi_j,X_{t_j}^B\>\Big)\Big]
&=\prod_{j=1}^p\E\Big[\exp\Big(-2\pi{\rm i}\Big\<X_{t_{\si_j}-t_{\si_{j-1}}}^B,\sum_{i=j}^p\xi_{\si_i}\Big\>\Big)\Big].
\end{split}\end{equation}
Note that, by the independence, the property of Brownian motion and \eqref{LT},
\begin{equation}\begin{split}\label{2LTU}
\E\Big[\exp\Big(-2\pi{\rm i}\Big\<X_{t_{\si_j}-t_{\si_{j-1}}}^B,\sum_{i=j}^p\xi_{\si_i}\Big\>\Big)\Big]
&=\E\Big[\exp\Big(-2\pi^2\Big|\sum_{i=j}^p\xi_{\si_i}\Big|^2S_{t_{\si_j}-t_{\si_{j-1}}}^B\Big)\Big]\\
&=\exp\Big[-B\Big(2\pi^2\Big|\sum_{i=j}^p\xi_{\si_i}\Big|^2\Big)(t_{\si_j}-t_{\si_{j-1}})\Big].
\end{split}\end{equation}
 Hence
\begin{equation*}\begin{split}
\E\Big[\exp\Big(-2\pi {\rm i }\sum_{j=1}^p\<\xi_j,X_{t_j}^B\>\Big)\Big]
=\prod_{j=1}^p\exp\Big[-B\Big(2\pi^2\Big|\sum_{i=j}^p\xi_{\si_i}\Big|^2\Big)(t_{\si_j}-t_{\si_{j-1}})\Big].
\end{split}\end{equation*}
By integrating both sides of the above equality over $\Delta_\si$, we have
\begin{equation}\begin{split}\label{3LTU}
&\int_{\Delta_\si}\E\Big[\exp\Big(-2\pi {\rm i}\sum_{j=1}^p\<\xi_j,X_{t_j}^B\>\Big)\Big]\,\d t_1\d t_2 \cdots \d t_{p}\\
&=\int_{\Delta_\si}\exp\Big[-\sum_{j=1}^p B\Big(2\pi^2\Big|\sum_{i=j}^p\xi_{\si_i}\Big|^2\Big)(t_{\si_j}-t_{\si_{j-1}})\Big]\,\d t_{\si_1}\d t_{\si_2} \cdots \d t_{\si_p}\\
&\le\int_{[0,t]^p}\exp\Big[-\sum_{j=1}^pB\Big(2\pi^2\Big|\sum_{i=j}^p\xi_{\si_i}\Big|^2\Big)s_j\Big]\,\d s_1\d s_2\cdots \d s_p\\
&\leq \prod_{j=1}^p\min\Big\{\ff{1}{B(2\pi^2|\sum_{i=j}^p\xi_{\si_i}|^2)},t\Big\}.
\end{split}\end{equation}

Thus, according to \eqref{1LTU} and \eqref{3LTU}, by employing \eqref{bound-B}, we have
\begin{equation}\begin{split}\label{3LTU'}
\Big|\E\Big[\prod_{j=1}^p\widehat{\mu_t^B}(\xi_j)\Big]\Big|
&\lesssim\ff 1 {t^p}\sum_{\si\in \mathcal{S}_p}\prod_{j=1}^p\min\Big\{\ff{1}{|\sum_{i=j}^p\xi_{\si_i}|^{2\aa}},t\Big\},
\end{split}\end{equation}
which proves the desired result when $H=1/2$.

\medskip

\underline{\textbf{Part 2}}. It remains to discuss the case when $H\neq 1/2$ under the condition in Theorem \ref{TH2}(ii).

It is clear that the process $(X_t^{B,H})_{t\ge 0}$ no longer has independent increments. So the equality \eqref{01LTU} is unavailable. To overcome the difficulty, we employ the local non-determinism of fBM $X^H$ introduced in \cite[Section 2.1]{X}, i.e.,
\begin{equation}\label{4LTU}
{\rm Cov}(X_{t_1}^H-X_{t_0}^H,X_{t_2}^H-X_{t_1}^H,\cdots,X_{t_p}^H-X_{t_{p-1}}^H)\gtrsim {\rm diag}(|t_1-t_0|^{2H},|t_2-t_1|^{2H},\cdots,|t_p-t_{p-1}|^{2H}),
\end{equation}
for any $0\le t_0< t_1<\cdots<t_p\le t$, where the left hand side is the $p\times p$ covariance matrix and the right hand side is the $p\times p$ diagonal matrix.  Then by \eqref{4LTU}, the independence of $(X_t^H)_{t\ge 0}$ and $(S_t^B)_{t\ge 0}$ and Lemma \ref{SDU}, we find some constants $C_1,C_2>0$ such that, for any $H\in(0,1/2)$,
\begin{equation}\begin{split}\label{5LTU}
&\int_{\Delta_\si}\E\Big[\exp\Big(-2\pi {\rm i}\sum_{j=1}^p\<\xi_j,X_{t_j}^{B,H}\>\Big)\Big]\,\d t_1\d t_2 \cdots \d t_{p}\\
&\le\int_{\Delta_\si}\E\exp\Big[-C_1\sum_{j=1}^p\Big|\sum_{i=j}^p\xi_{\si_i}\Big|^2\Big(S_{t_{\si_j}}^B-S_{t_{\si_{j-1}}}^B\Big)^{2H}\Big]\,\d t_{\si_1}\d t_{\si_2} \cdots \d t_{\si_p}\\
&\lesssim \int_{\Delta_\si}\prod_{j=1}^p \exp\Big(-C_2\Big|\sum_{i=j}^p\xi_{\si_i}\Big|^{\ff{2\aa}{(1-2H)\aa+2H}}(t_{\si_j}-t_{\si_{j-1}})^{\ff{2H}{(1-2H)\aa+2H}}\Big)\,\d t_{\si_1}\d t_{\si_2} \cdots \d t_{\si_p}\\
&\lesssim \prod_{j=1}^p\int_{[0,t]}\exp\Big(-C_2\Big|\sum_{i=j}^p\xi_{\si_i}\Big|^{\ff{2\aa}{(1-2H)\aa+2H}}s_j^{\ff{2H}{(1-2H)\aa+2H}}\Big)\d s_1\d s_2\cdots \d s_p\\
&\lesssim\prod_{j=1}^p\min\Big\{\Big|\sum_{i=j}^p\xi_{\si_i}\Big|^{-\aa/H},t\Big\},
\end{split}\end{equation}
where in the equality we have used the fact that $(S_t^B)_{t\ge 0}$ has stationary, independent increments.
By the same argument as in \eqref{5LTU}, for every $H\in(1/2,1)$, we also have
\begin{equation}\begin{split}\label{6LTU}
\int_{\Delta_\si}\E\Big[\exp\Big(-2\pi {\rm i}\sum_{j=1}^p\<\xi_j,X_{t_j}^B\>\Big)\Big]\,\d t_1\d t_2 \cdots \d t_{p}
\lesssim\prod_{j=1}^p\min\Big\{\Big|\sum_{i=j}^p\xi_{\si_i}\Big|^{-\aa/H},t\Big\}.
\end{split}\end{equation}

Therefore, combining \eqref{3LTU'} with \eqref{5LTU} and \eqref{6LTU}, we finish the proof Lemma \ref{LTU}.
\end{proof}

In general, we give the upper bound on $\E[\|\nn u_\vv\|_{L^p(\T^d)}^p]$ with $p\in\N$ even. To this end, we need the following generalized version of Young's convolution inequality adapted from \cite[Lemma 3.5]{HMT}, which can be proved directly by applying H\"{o}lder's inequality and Young's convolution inequality.
\begin{lem}\label{G}
Let $p\in\mathbb{N}$, $p\ge 2$, $f_1,f_2,\cdots,f_p:\mathbb{Z}^d\to[0,\infty]$, $g_2,g_3,\cdots,g_p:\mathbb{Z}^d\to[0,\infty]$ be measurable functions and choose $\ll_i,\Lambda_j\in[1,\infty]$, $i\in\{1,2,\cdots, p\}$ and $j\in\{2,3,\cdots,p\}$ such that
\begin{align*}
\begin{cases}
\ff 1 {\ll_1}+\ff 1 {\ll_2}+\ff 1 {\Lambda_2}=2,&\\
\ff 1 {\ll_k}+\ff 1 {\Lambda_k}=1,\quad 3\le k\le p.
\end{cases}
\end{align*}
Then
$$\|T_p\|_{l^1(\mathbb{Z}^{d\times p})}\le\prod_{i=1}^p\|f_i\|_{l^{\ll_i}(\mathbb{Z}^d)}\prod_{j=2}^p\|g_j\|_{l^{\Lambda_j}(\mathbb{Z}^d)},$$
where
$$T_p(\xi_1,\cdots,\xi_p):=\prod_{i=1}^p f_i(\xi_i)\prod_{j=2}^p g_j\Big(\sum_{i=1}^j\xi_i\Big),\quad (\xi_1,\cdots,\xi_p)\in(\mathbb{Z}^d)^p.$$
\end{lem}

\begin{prp}\label{NUP}
Under the Assumptions in Theorem \ref{TH2}, if $u_\vv$ is the solution to the Poisson's equation \eqref{PoisE-H}, then for any $p\in\N$ even and large enough $t>0$,
\begin{equation}\label{UU}
\E[\|\nn u_\vv\|_{L^p(\T^d)}^p]\lesssim \vv^{\ff p 2},
\end{equation}
where
\begin{equation}\label{vv}\vv=\begin{cases}
t^{-1},\quad &d<2+ \aa/ H,\\
t^{-1}\log t,\quad & d=2+ \aa /H,\\
t^{-\ff 2 {d-\aa/ H}},\quad & d>2+\aa /H.
\end{cases}\end{equation}
\end{prp}
\begin{proof}
Let $p$ be a positive even integer.  Combining \eqref{FM}, \eqref{GRA}, \eqref{FP} and Lemma \ref{LTU} together, we have
\begin{equation}\begin{split}\label{NUVP}
&\E[\|\nn u_\vv\|_{L^p(\T^d)}^p]\lesssim\sum_{\xi_1,\cdots,\xi_p\in\mathbb{Z}^d\backslash\{0\}}\prod_{i=1}^p\ff{\exp(-\vv|\xi_i|^2/2)}{|\xi_i|}
\Big|\E\Big[\prod_{i=1}^p\widehat{\mu_t^{B,H}}(\xi_i)\Big]\Big|\dd_0\Big(\sum_{i=1}^p\xi_i\Big)\\
&\lesssim\ff 1 {t^p}\sum_{\xi_1,\cdots,\xi_p\in\mathbb{Z}^d\backslash\{0\}}\prod_{i=1}^p\ff{\exp(-\vv|\xi_i|^2/2)}{|\xi_i|}\sum_{\si\in\mathcal{S}_p}\prod_{j=1}^p\min\Big\{ \Big|\sum_{i=j}^p\xi_{\si_i}\Big|^{-\aa/H},t\Big\}\dd_0\Big(\sum_{i=1}^p\xi_i\Big)\\
&=\ff 1{t^{p-1}}\sum_{\xi_1,\cdots,\xi_p\in\mathbb{Z}^d\backslash\{0\}}\prod_{i=1}^p\ff{\exp(-\vv|\xi_i|^2/2)}{|\xi_i|}\sum_{\si\in\mathcal{S}_p}\prod_{j=2}^p\min\Big\{ \Big|\sum_{i=1}^{j-1}\xi_{\si_i}\Big|^{-\aa/H},t\Big\}\dd_0\Big(\sum_{i=1}^p\xi_{\si_i}\Big)\\
&=\ff 1{t^{p-1}}\sum_{\xi_1,\cdots,\xi_p\in\mathbb{Z}^d\backslash\{0\}}\prod_{i=1}^p\ff{\exp(-\vv|\xi_i|^2/2)}{|\xi_i|}\sum_{\si\in\mathcal{S}_p}\prod_{j=1}^{p-1}\min\Big\{ \Big|\sum_{i=1}^{j}\xi_{\si_i}\Big|^{-\aa/H},t\Big\}\dd_0\Big(\sum_{i=1}^p\xi_{\si_i}\Big),
\end{split}\end{equation}
where the first equality is due to that $\sum_{j=1}^p\xi_{\sigma_j}=\sum_{j=1}^p\xi_j=0$.

We begin by discussing the simplest case when $p=2$. According to \eqref{NUVP} and Lemma \ref{g}, we have
\begin{equation*}\begin{split}
\E[\|\nn u_\vv\|_{L^2(\T^d)}^2]&\lesssim \ff 1 t\sum_{\xi_1,\xi_2\in\Z^d\backslash\{0\}}\prod_{i=1}^2\ff{\exp(-\vv|\xi_i|^2/2)}{|\xi_i|}\sum_{\si\in\mathcal{S}_2}
\min\big\{|\xi_{\si_i}|^{-\aa/H},t \big\}\dd_0\Big(\sum_{i=1}^2\xi_{\si_i}\Big)\\
&\lesssim\ff 1 t \sum_{\xi\in\Z^d\backslash\{0\}}\ff{\exp(-\vv|\xi|^2)}{|\xi|^{2+ \aa/ H}}\lesssim\ff 1 t\|\phi_{c\varepsilon}\|_{l^{2+\ff \aa H}(\Z^d)}^{2+\ff \aa H}
\lesssim \vv,
\end{split}\end{equation*}
where $\vv$ as defined in \eqref{vv}.

It remains to verify that \eqref{UU} holds for $p\ge 4$. Set
$$\Xi:=\Big\{(\xi_1,\cdots,\xi_p)\in(\mathbb{Z}^d\backslash\{0\})^p:\  \sum_{i=1}^j \xi_i\neq 0,j=1,2,\cdots,p-1 \Big\}.$$
Without loss of generality, we only consider the case when $\si$ is the identity permutation. Then
\begin{equation}\begin{split}\label{XL}
&\sum_{\xi_1,\cdots,\xi_p\in\Xi}\prod_{i=1}^p\ff{\exp(-\vv|\xi_i|^2/2)}{|\xi_i|}\prod_{j=1}^{p-1} \Big|\sum_{i=1}^j\xi_{i}\Big|^{-\aa/H}\dd_0\Big(\sum_{i=1}^p\xi_i\Big)\\
&\lesssim\sum_{\xi_1,\cdots,\xi_p\in \mathbb{Z}^d}\prod_{i=1}^p\ff{\exp(-\vv|\xi_i|^2/2)}{|\xi_i|+1}\prod_{j=1}^{p-1} \Big(\Big|\sum_{i=1}^j\xi_{i}\Big|+1\Big)^{-\aa/H}\dd_0\Big(\sum_{i=1}^p\xi_i\Big)\\
&\lesssim\sum_{\xi_1,\cdots,\xi_p\in \mathbb{Z}^d}\prod_{i=1}^p\ff{\exp\Big(-c\vv|\xi_i|^2\Big)}{|\xi_i|+1}
\prod_{j=1}^{p-1}\ff{\exp\Big(-\ff{c\alpha\vv}{H}  \Big|\sum_{i=1}^j\xi_i\Big|^2\Big)}
{\Big(\Big|\sum_{i=1}^j\xi_{i}\Big|+1\Big)^{\aa/H}}\dd_0\Big(\sum_{i=1}^p\xi_i\Big),
\end{split}\end{equation}
for some positive constant $c$,
 the first one is due to that $a\geq(a+1)/2$ for every $a\geq1$, and the third one follows from that
\begin{equation}\label{XU}
\prod_{i=1}^p\exp(-\vv|\xi_i|^2/2)\le\Big\{\prod_{i=1}^p\exp\Big(-c\vv|\xi_i|^2\Big)\Big\}
\prod_{j=1}^{p-1}\exp\Big(-\ff{c\alpha\vv} H\Big|\sum_{i=1}^j\xi_i\Big|^2\Big),
\end{equation}
for some constant $c>0$. Indeed, by the elementary inequality
$$\Big|\sum_{i=1}^j\xi_i\Big|^2\le j\sum_{i=1}^j|\xi_i|^2,\quad j\in\N,$$
we have
$$\ff{\alpha} H\sum_{j=1}^{p-1}\Big|\sum_{i=1}^j\xi_i\Big|^2+\sum_{i=1}^p|\xi_i|^2\le\Big[\ff{\aa p(p-1)}{2H} +1\Big]\sum_{i=1}^p|\xi_i|^2,$$
which immediately implies that \eqref{XU} holds for some constant $c\le (\alpha p(p-1)/H+2)^{-1}$.

Letting
\begin{align*}\begin{cases}
f_1=\phi^{1+\aa/H}_{c\varepsilon},\quad f_2=\cdots=f_{p-1}=\phi_{c\varepsilon},&\\
g_2=\cdots=g_{p-2}=\phi^{\aa/H}_{c\varepsilon},\quad g_{p-1}=\phi^{1+\aa/H}_{c\varepsilon},\quad p\geq 3,\\
\end{cases}\end{align*}
and choosing $\ll_1,\cdots,\ll_{p-1},\Lambda_2,\cdots,\Lambda_{p-1}$ from $[1,\infty]$ such that
\begin{equation}\begin{split}\label{lL}\begin{cases}
\ff 1 {\ll_1}+\ff 1 {\ll_2}+\ff 1 {\Lambda_2}=2,&\\
\ff 1 {\ll_k}+\ff 1 {\Lambda_k}=1,&\quad 3\le k\le p-1,
\end{cases}\end{split}\end{equation}
by Lemma \ref{G}, we have
\begin{equation}\begin{split}\label{XL1}
&\sum_{\xi_1,\cdots,\xi_p\in \mathbb{Z}^d}\prod_{i=1}^p\ff{\exp\Big(-c\vv|\xi_i|^2\Big)}{|\xi_i|+1}
\prod_{j=1}^{p-1}\ff{\exp\Big(-\ff{c\alpha\vv} H   \Big|\sum_{i=1}^j\xi_i\Big|^2\Big)}
{\Big(\Big|\sum_{i=1}^j\xi_i\Big|+1\Big)^{\aa/H}}\dd_0\Big(\sum_{i=1}^p\xi_i\Big)\\
&=\sum_{\xi_1,\cdots,\xi_p\in \mathbb{Z}^d} \phi_{c\varepsilon}(\xi_1)\Big\{\prod_{i=2}^{p-1}\phi_{c\varepsilon}(\xi_i)\Big\}
\phi_{c\varepsilon}(\xi_p)\phi_{c\varepsilon}^{\aa/H}(\xi_1)\Big\{\prod_{j=2}^{p-2}\phi_{c\varepsilon}^{\aa/H}\Big(\sum_{i=1}^j\xi_i\Big)\Big\}\phi_{c\varepsilon}^{\aa/H}\Big(\sum_{i=1}^{p-1}\xi_i\Big)\dd_0\Big(\sum_{i=1}^p\xi_i\Big)\\
&=\sum_{\xi_1,\cdots,\xi_p\in \mathbb{Z}^d}\phi_{c\varepsilon}^{1+\aa/H}(\xi_1)\Big\{\prod_{i=2}^{p-1}\phi_{c\varepsilon}(\xi_i)\Big\}
\Big\{\prod_{j=2}^{p-2}\phi_{c\varepsilon}^{\aa/H}\Big(\sum_{j=1}^j\xi_i\Big)\Big\}
\phi_{c\varepsilon}^{1+\aa/H}\Big(\sum_{i=1}^{p-1}\xi_i\Big)\\
&\le\|\phi_{c\varepsilon}^{1+\aa/H}\|_{l^{\ll_1}(\Z^d)}\Big\{\prod_{i=2}^{p-1}\|\phi_{c\varepsilon}\|_{l^{\ll_i}(\Z^d)}\Big\}
\Big\{\prod_{j=2}^{p-2}\|\phi_{c\varepsilon}^{\aa/H}\|_{l^{\Lambda_j}(\Z^d)}\Big\}\|\phi_{c\varepsilon}^{1+\aa/H}\|_{l^{\Lambda_{p-1}}(\Z^d)}\\
&=:I,
\end{split}\end{equation}
where in the second equality we have used the fact that
$$\phi_{\varepsilon}(\xi_p)=\phi_{\varepsilon}\Big(-\sum_{i=1}^{p-1}\xi_i\Big)
=\phi_{\varepsilon}\Big(\sum_{i=1}^{p-1}\xi_i\Big),\quad\varepsilon>0.$$

Now we prove the assertion case by case. In the sequel, we keep in mind that $t$ is big and hence $\varepsilon$ is small by the choice of \eqref{vv}.

\textbf{(1)} Let $d<2+\aa/H$. We split this case into two parts.

\textbf{Part (i)} Let $d=1$. By taking
\begin{align*}\begin{cases}
\ll_1=1,\\
 \ll_k=1,\,\Lambda_k=\infty,\quad 2\le k\le p-1,
\end{cases}\end{align*}
we easily see that the condition \eqref{lL} is satisfied. Moreover,
$$I=\|\phi_{c\varepsilon}^{1+\aa/H}\|_{l^1(\Z)}\Big(\prod_{i=2}^{p-1}\|\phi_{c\varepsilon}\|_{l^1(\Z)}\Big)
\Big(\prod_{j=2}^{p-2}\|\phi_{c\varepsilon}^{\aa/H}\|_{l^{\infty}(\Z)}\Big)\|\phi_{c\varepsilon}^{1+\aa/H}\|_{l^\infty(\Z)}.$$
According to Lemma \ref{g}, we obtain that
\begin{align*}\begin{cases}
\|\phi_{c\varepsilon}^{1+\aa/H}\|_{l^1(\Z)}=\|\phi_{c\varepsilon}\|_{l^{1+\aa/H}(\Z)}^{1+\aa/H}\lesssim 1,\quad
\|\phi_{c\varepsilon}\|_{l^1(\Z)}\lesssim |\log \vv|,\\
\|\phi_{c\varepsilon}^{\aa/H}\|_{l^\infty(\Z)}\lesssim 1,\quad
\|\phi_{c\varepsilon}^{1+\aa/H}\|_{l^\infty(\Z)}\lesssim 1.
\end{cases}\end{align*}
By choosing $\vv$ as in \eqref{vv}, we arrive at
\begin{equation}\label{I30}
I\lesssim |\log\vv|^{p-2}\lesssim |\log t|^{p-2}\lesssim t^{p-1}\vv^{\ff p 2}.
\end{equation}

\textbf{Part (ii)} Let $d>1$. We claim that there exists $1<\eta<d$ such that
\begin{equation}\label{BD}d<\ff{\aa\eta}{H(\eta-1)}<\ff{\eta}{\eta-1}\Big(1+\ff{\aa}H\Big).\end{equation}
In fact, it is enough to verify the first inequality since the second one is trivial. If $d\le \aa/H$, then the first inequality is obviously true. If $d>\aa/H$, the first inequality is equivalent to $\eta<d/(d-\aa/H)$. Noticing that $d/(d-\aa/H)>d/2$ since $d<2+\aa/H$, we can choose $\eta$ in the following way, i.e.,
\begin{align*}\begin{cases}
\max\left\{1,\ff d 2\right\}<\eta<d,&\quad d\le \aa/H,\\
\max\left\{1,\ff d 2\right\}<\eta<\ff d {d-\aa/H},&\quad d>\aa/H.
\end{cases}\end{align*}
Letting
\begin{align*}\begin{cases}
\ll_1=1,\\
\ll_k=\eta,\,\Lambda_k=\ff{\eta}{\eta-1},\quad 2\le k\le p-1,
\end{cases}\end{align*}
we have
$$I=\|\phi_{c\varepsilon}^{1+\aa/H}\|_{l^1(\Z^d)}\Big(\prod_{i=2}^{p-1}\|\phi_{c\varepsilon}\|_{l^{\eta}(\Z^d)}\Big)
\Big(\prod_{j=2}^{p-2}\|\phi_{c\varepsilon}^{\aa/H}\|_{l^{\eta/(\eta-1)}(\Z^d)}\Big)\|\phi_{c\varepsilon}^{1+\aa/H}\|_{l^{\eta/(\eta-1)}(\Z^d)}.$$

(ii.1) If $d<1+\aa/H$, then by Lemma \ref{g} and \eqref{BD}, we deduce that
\begin{align*}\begin{cases}
\|\phi_{c\varepsilon}^{1+\aa/H}\|_{l^1(\Z^d)}=\|\phi_{c\varepsilon}\|_{l^{1+\aa/H}(\Z^d)}^{1+\aa/H}\lesssim 1,&\\
\|\phi_{c\varepsilon}\|_{l^{\eta}(\Z^d)}\lesssim \vv^{-\ff 1 2(d/{\eta}-1)},&\\
\|\phi_{c\varepsilon}^{\aa/H}\|_{l^{\eta/(\eta-1)}(\Z^d)}=\|\phi_{c\varepsilon}\|_{l^{\aa\eta/(H(\eta-1))}(\Z^d)}^{\aa/H}\lesssim 1,&\\
\|\phi_{c\varepsilon}^{1+\aa/H}\|_{l^{\eta/(\eta-1)}(\Z^d)}=\|\phi_{c\varepsilon}\|_{l^{\eta(1+\aa/H)/(\eta-1)}(\Z^d)}^{1+\aa/H}\lesssim 1.
\end{cases}\end{align*}
Hence, \eqref{vv} and $d<2\eta$ imply that
\begin{equation}\label{I3}
I\lesssim \vv^{-\ff {p-2}2(d/\eta-1)}\lesssim t^{\ff {p-2}2(d/\eta-1)}\lesssim t^{p-1}\vv^{\ff p 2}.
\end{equation}

(ii.2) If $1+\aa/H<d<2+\aa/H$, then by Lemma \ref{g} and \eqref{BD}, we have
\begin{align*}\begin{cases}
\|\phi_{c\varepsilon}^{1+\aa/H}\|_{l^1(\Z^d)}=\|\phi_{c\varepsilon}\|_{l^{1+\aa/H}(\Z^d)}^{1+\aa/H}\lesssim\vv^{-\ff 1 2(d-1-\aa/H)},&\\
\|\phi_{c\varepsilon}\|_{l^{\eta}(\Z^d)}\lesssim\vv^{-\ff 1 2\left(d/\eta-1\right)},&\\
\|\phi_{c\varepsilon}^{\aa/H}\|_{l^{\eta/(\eta-1)}(\Z^d)}=\|\phi_{c\varepsilon}\|_{l^{\aa\eta/(H(\eta-1))}(\Z^d)}^{\aa/H}\lesssim 1,&\\
\|\phi_{c\varepsilon}^{1+\aa/H}\|_{l^{\eta/(\eta-1)}(\Z^d)}=\|\phi_{c\varepsilon}\|_{l^{\eta(1+\aa/H)/(\eta-1)}(\Z^d)}^{1+\aa/H}\lesssim 1.&\\
\end{cases}\end{align*}
Thus, according to that
$$\Big(\ff d {\eta}-1\Big)(p-2)+d-1-\ff \aa H<p-2,$$
we obtain
\begin{equation}\begin{split}\label{I4}
I
\lesssim t^{\ff 1 2\left[(d/\eta-1)(p-2)+d-1-\aa/H\right]}
\lesssim t^{\ff p 2 -1}= t^{p-1}\vv^{\ff p 2},
\end{split}\end{equation}
since $\vv=1/t$ by \eqref{vv}.

(ii.3) If $d=1+\aa/H$, then by Lemma \ref{g} and \eqref{BD}, we have
\begin{align*}\begin{cases}
\|\phi_{c\varepsilon}^{1+\aa/H}\|_{l^1(\Z^d)}=\|\phi_{c\varepsilon}\|_{l^{1+\aa/H}(\Z^d)}^{1+\aa/H}\lesssim |\log\vv|,&\\
\|\phi_{c\varepsilon}\|_{l^{\eta}(\Z^d)}\lesssim \vv^{{-\ff 1 2}(d/\eta-1)},&\\
\|\phi_{c\varepsilon}^{\aa/H}\|_{l^{\eta/(\eta-1)}(\Z^d)}=\|\phi_{c\varepsilon}\|_{l^{\aa\eta/(H(\eta-1))}(\Z^d)}^{\aa/H}\lesssim 1,&\\
\|\phi_{c\varepsilon}^{1+\aa/H}\|_{l^{\eta/(\eta-1)}(\Z^d)}=\|\phi_{c\varepsilon}\|_{l^{\eta(1+\aa/H)/(\eta-1)}(\Z^d)}^{1+\aa/H}\lesssim 1.
\end{cases}\end{align*}
Thus, by \eqref{vv} and $d/\eta<2$,
\begin{equation}\begin{split}\label{I5}
I
\lesssim t^{\ff 1 2(d/\eta-1)(p-2)}\lesssim t^{p-1}\vv^{\ff p 2}.
\end{split}\end{equation}

\textbf{(2)} Let $d=2+\aa/H$. By choosing
\begin{align*}\begin{cases}
\ll_1=1, \\
\ll_k=d,\,\Lambda_k=\ff d {d-1}, \quad 2\le k\le p-1,
\end{cases}\end{align*}
such that the condition \eqref{lL} is satisfied, we have
$$I=\|\phi_{c\varepsilon}^{1+\aa/H}\|_{l^1(\Z^d)}\Big(\prod_{i=2}^{p-1}\|\phi_{c\varepsilon}\|_{l^d(\Z^d)}\Big)
\Big(\prod_{j=2}^{p-2}\|\phi_{c\varepsilon}^{\aa/H}\|_{l^{d/(d-1)}(\Z^d)}\Big)\|\phi_{c\varepsilon}^{1+\aa/H}\|_{l^{d/(d-1)}(\Z^d)}.$$
Noting that $1+\aa/H=d-1$, by Lemma \ref{g} and a simple computation, we obtain
\begin{align*}\begin{cases}
\|\phi_{\varepsilon/c}^{1+\aa/H}\|_{l^1(\Z^d)}=\|\phi_{c\varepsilon}\|_{l^{1+\aa/H}(\Z^d)}^{1+\aa/H}\lesssim \vv^{-\ff 1 2 \left(\ff{d}{1+\aa/H}-1\right)(1+\aa/H)}=\vv^{-\ff 1 2},&\\
\|\phi_{c\varepsilon}\|_{l^d(\Z^d)}\lesssim|\log \vv|^{1/d},&\\
\|\phi_{c\varepsilon}^{\aa/H}\|_{l^{d/(d-1)}(\Z^d)}=\|\phi_{c\varepsilon}\|_{l^{d(d-2)/(d-1)}(\Z^d)}^{d-2}\lesssim \vv^{-\ff  1 2\left[(d-1)/(d-2)-1\right](d-2)}=\vv^{-\ff 1 2},&\\
\|\phi_{c\varepsilon}^{1+\aa/H}\|_{l^{d/(d-1)}(\Z^d)}=\|\phi_{c\varepsilon}\|_{l^{d(1+\aa/H)/(d-1)}(\Z^d)}^{1+\aa/H}\lesssim|\log \vv|^{1-1/d},
\end{cases}\end{align*}
which together with the choice of $\varepsilon= \log t/t$ from \eqref{vv} implies that
\begin{equation}\begin{split}\label{I1}
I
\lesssim\vv^{-\ff 1 2 (p-2)}|\log \vv|^{1+\ff{p-3}d}
\lesssim t^{\ff 1 2(p-2)}(\log t)^{\ff{p-3}d+2-\ff p 2}
\lesssim t^{p-1}\vv^{\ff p 2},
\end{split}\end{equation}
where the last inequality is due to that $p\ge 2-1/(d-1)$ since $d=2(1+\aa)$ and $p$ is an even number in $\N$.

\textbf{(3)} Let $d>2+\aa/H$. By choosing $\eta<d$ such that $1+\aa/H<\eta-1<d$, $\eta(1+\aa/H)/(\eta-1)<d$ and
\begin{align*}\begin{cases}
\ll_1=1,\\
\ll_k=\eta,\,\Lambda_k=\ff{\eta}{\eta-1}, \quad 2\le k\le p-1,
\end{cases}\end{align*}
it is easily to verify that \eqref{lL} holds. Then
$$I=\|\phi_{c\varepsilon}^{1+\aa/H}\|_{l^1(\Z^d)}\Big(\prod_{i=2}^{p-1}\|\phi_{c\varepsilon}\|_{l^{\eta}(\Z^d)}\Big)
\Big(\prod_{j=2}^{p-2}\|\phi_{c\varepsilon}^{\aa/H}\|_{l^{\eta/(\eta-1)}(\Z^d)}\Big)
\|\phi_{c\varepsilon}^{1+\aa/H}\|_{l^{\eta/(\eta-1)}(\Z^d)}.$$
According to Lemma \ref{g} and the choice of $\eta$, we arrive at
\begin{align*}\begin{cases}
\|\phi_{c\varepsilon}^{1+\aa/H}\|_{l^1(\Z^d)}=\|\phi_{c\varepsilon}\|_{l^{1+\aa/H}(\Z^d)}^{1+\aa/H}\lesssim \vv^{-\ff 1 2\left(\ff d {1+\aa/H}-1\right)(1+\aa/H)}=\vv^{-\ff 1 2(d-1-\aa/H)},&\\
\|\phi_{c\varepsilon}\|_{l^{\eta}(\Z^d)}\lesssim \vv^{-\ff 1 2( d /\eta-1)},&\\
\|\phi_{c\varepsilon}^{\aa/H}\|_{l^{\eta/(\eta-1)}(\Z^d)}=\|\phi_{c\varepsilon}\|_{l^{\aa\eta/(H(\eta-1))}(\Z^d)}^{\aa/H}
\lesssim \vv^{-\ff 1 2 \left[d\left(1- 1 /\eta\right)-\aa/H\right]},&\\
\|\phi_{c\varepsilon}^{1+\aa/H}\|_{{l^{\eta/(\eta-1)}}(\Z^d)}=\|\phi_{c\varepsilon}\|_{l^{(1+\aa/H)\eta/(\eta-1)}(\Z^d)}^{1+\aa/H}
\lesssim \vv^{-\ff 1 2\left[d(1-1/\eta)-1-\aa/H\right]}.
\end{cases}\end{align*}
Thus, together with \eqref{vv}, it is easily to see that
\begin{equation}\label{I2}
I\lesssim \vv^{-\ff 1 2[(p-1)(d-\aa/H)-p]}\lesssim t^{p-1}\vv^{\ff p 2}.
\end{equation}

Gathering \eqref{NUVP}, \eqref{XL} and \eqref{XL1} together, we have
\begin{equation}\begin{split}\label{U-}
\E[\|\nn u_\vv\|_{L^p(\T^d)}^p]\lesssim\ff 1 {t^{p-1}}\sum_{\xi_1,\cdots,\xi_p\in\Xi}\prod_{i=1}^p\ff{\exp(-\vv|\xi_i|^2/2)}{|\xi_i|}\prod_{j=1}^{p-1}
\Big|\sum_{i=1}^j\xi_i\Big|^{-2\aa}\dd_0\Big(\sum_{i=1}^p\xi_i\Big)
\lesssim t^{1-p} I.
\end{split}\end{equation}
Thus, by \eqref{I30}, \eqref{I3}, \eqref{I4}, \eqref{I5}, \eqref{I1} and \eqref{I2},  we immediately obtain
\begin{equation}\label{U}
\E[\|\nn u_\vv\|_{L^p(\T^d)}^p]\lesssim\vv^{\ff p 2}.
\end{equation}

Let
$$q:=\min\Big\{j\in\{1,2,\cdots,p-1\}:\  \xi_i\in\mathbb{Z}^d\setminus\{0\},\, i=1,\cdots,j,\,\sum_{i=1}^j \xi_i=0\Big\}.$$
Then, by the methods of induction, similar as the derivation of \eqref{U-}, \eqref{U} and \eqref{NUVP}, we finally arrive at
\begin{align*}
&\E[\|\nn u_\vv\|_{L^p(\T^d)}^p]\\
&\lesssim\sum_{\xi_1,\cdots,\xi_q\in\Xi}  \prod_{i=1}^q  \ff{\exp(-\vv|\xi_i|^2/2)}{|\xi_i|}\prod_{j=1}^{q-1}\Big|\sum_{i=1}^j\xi_i\Big|^{-\aa/H}\dd_0
\Big(\sum_{i=1}^q\xi_i\Big)\times\\
&\quad \sum_{\xi_{q+1},\cdots,\xi_p\in\Z^d\backslash\{0\}}\prod_{i=q+1}^p\ff{\exp(-\vv|\xi_i|^2/2)}{|\xi_i|}\prod_{j=q+1}^{p-1}\min\Big\{ \Big|\sum_{i={q+1}}^j\xi_i\Big|^{-\aa/H}, t\Big\}\dd_0\Big(\sum_{i=q+1}^p\xi_i\Big)\\
&\lesssim \vv^{\ff q 2}\vv^{\ff{p-q} 2}\lesssim \vv^{\ff p 2},
\end{align*}
which completes the proof.
\end{proof}

\begin{proof}[Proof of Theorem \ref{TH2}]
For every  $p\ge1$,
we may take a even number $q\in\mathbb{N}$ such that $q\geq p$. According to \eqref{WPU} and \eqref{UU}, by the monotonicity of $p\mapsto\W_p$, we have
$$\E[\W_p^p(\mu_t^{B,H},\mathfrak{m})]\lesssim \vv^{\ff p 2}+\Big(\E[\|\nn u_{\vv}\|_{L^q(\T^d)}^q]\Big)^{p/q}\lesssim\vv^{\ff p 2},$$
where $\vv$ is also given by \eqref{vv} for large enough $t>0$. 
\end{proof}



\section{Lower bounds}
In this section, we first prove Theorem \ref{TH1}, and then we establish a lower bound for $\R^d$-valued sfBMs which can be regard as an extension of our approach from the compact setting to the non-compact one.
\subsection{Proofs of Theorem \ref{TH1}: the sBM case}
To show the lower estimate in Theorem \ref{TH1}, we also need the matching lower bound on $\E[|\widehat{\mu_t^B}(\xi)|^2]$, which is presented next.
\begin{lem}\label{XLT}
Assume that $B\in\mathbf{B}$. Then for any $\xi\in\mathbb{Z}^d\backslash\{0\}$, it holds that
$$\E\Big[|\widehat{\mu_t^B}(\xi)|^2\Big]\asymp \ff 1 {tB(2\pi^2|\xi|^2)},$$
for large enough $t>0$.
\end{lem}
\begin{proof}
Firstly, we show the upper bound. Taking $p=2$ and $\xi_1=-\xi_2=\xi\in\Z^d\backslash\{0\}$ in \eqref{3LTU'}, we have
\begin{equation}\begin{split}\label{1XLT}
\E\Big[|\widehat{\mu_t^B}(\xi)|^2\Big]&=\Big|\E\Big[\prod_{j=1}^2\widehat{\mu_t^B}(\xi_j)\Big]\Big|\le\ff 1 {t^2}\sum_{\si\in\mathcal{S}_2}\prod_{j=1}^2\min\left\{\ff 1 {B(2\pi^2|\sum_{i=j}^2\xi_{\si_i}|^2)},t\right\}\\
&=\ff 1 t \sum_{\sigma\in\mathcal{S}_2}\min\left\{\ff 1 {B(2\pi^2|\xi_{\si_2}|^2)},t\right\}
\leq\ff 2 {tB(2\pi^2|\xi|^2)},\quad t>0,
\end{split}\end{equation}
 where the second equality is due to that $\xi_{\si_1}+\xi_{\si_2}=0$ for any $\sigma\in\mathcal{S}_2$ and $B\in\mathbf{B}$.

Now, we turn to prove the lower bound. By the fact that $|\widehat{\mu_t^B}(\xi)|^2=\widehat{\mu_t^B}(\xi)\widehat{\mu_t^B}(-\xi)$, we may divide $\E[|\widehat{\mu_t^B}(\xi)|^2]$ into two parts.
\begin{equation}\begin{split}\label{2XLT}
\E[|\widehat{\mu_t^B}(\xi)|^2]
&=\ff 1 {t^2}\int_0^t\int_0^{s_1}\E\left[\exp(-2\pi^2|\xi|^2(S_{s_1}^B-S_{s_2}^B))\right]\,\d s_2 \d s_1\\
&+\ff 1 {t^2}\int_0^t\int_{s_1}^t\E\left[\exp\big(-2\pi^2|\xi|^2(S_{s_2}^B-S_{s_1}^B)\big)\right]\,\d s_2 \d s_1\\
&=\ff 1 {t^2}\int_0^t\int_0^{s_1}\exp\left(-B(2\pi^2|\xi|^2)(s_1-s_2)\right)\,\d s_2 \d s_1\\
&+\ff 1 {t^2}\int_0^t\int_{s_1}^t\exp\left(-B(2\pi^2|\xi|^2)(s_2-s_1)\right)\,\d s_2 \d s_1\\&=:I_1+I_2,\quad t>0.
\end{split}\end{equation}
For $I_1$, since $B\in\mathbf{B}$, it is easy to see that for every $t>1$,
\begin{equation}\begin{split}\label{3XLT}
I_1&=\ff 1 {t^2}\int_0^t\int_0^{s_1}\exp\big(-B(2\pi^2|\xi|^2) u\big)\,\d u\d s_1\\
&\ge\ff 1 {t^2}\int_{\ff t 2}^t\int_0^{\ff t 2}\exp\big(-B(2\pi^2|\xi|^2) u\big)\,\d u\d {s_1}\\
&\gtrsim\ff 1 {tB(2\pi^2|\xi|^2)},\quad \xi\in\Z^d\backslash\{0\}.
\end{split}\end{equation}
For $I_2$, similarly, we have
\begin{equation}\begin{split}\label{4XLT}
I_2&=\ff 1{t^2}\int_0^t\int_0^{t-s_1}\exp\big(-B(2\pi^2|\xi|^2) u\big)\,\d u\d s_1\\
&\geq \ff 1{t^2}\int_0^{\frac{t}{2}}\int_0^{\frac{t}{2}}\exp\big(-B(2\pi^2|\xi|^2) u\big)\,\d u\d s_1\\
&\gtrsim\ff 1 {tB(2\pi^2|\xi|^2)},\quad \xi\in\Z^d\backslash\{0\},\,t>1.
\end{split}\end{equation}

Combining \eqref{1XLT} with \eqref{2XLT}, \eqref{3XLT} and \eqref{4XLT} together, we complete the proof.
\end{proof}

Now, according to \eqref{W1}, by the estimate on $\E[|\widehat{\mu_t^B}(\xi)|^2]$ in Lemma \ref{XLT}, we are ready to prove the lower bound of $\E[\W_1(\mu_t^B,\mathfrak{m})]$.
\begin{prp}\label{sBM-main}
 Assume that $B\in\mathbf{B}_\aa\cap\mathbf{B}^\aa$ for some $\aa\in[0,1]$.  Then
\begin{equation}\label{MW1L}
\E[\W_1(\mu_t^B,\mathfrak{m})]\gtrsim
\begin{cases}
\ff 1 {t^{1/2}},\quad & d<2(1+\aa),\\
\sqrt{\ff{\log t}t},\quad & d=2(1+\aa),\\
t^{-\ff 1 {d-2\aa}},\quad & d>2(1+\aa),
\end{cases}
\end{equation}
for large enough $t>0$.
\end{prp}

\begin{proof}
For every $\varepsilon>0$, let $u_\varepsilon$ be the solution to the following Poisson's equation regularized by the heat semigroup $(P_\vv)_{\vv>0}$ on $\T^d$:
\begin{equation}\label{PE}-\Delta u=P_\vv(\mu_t^B-\mathfrak{m}).\end{equation}
By \eqref{F2}, \eqref{FM}, \eqref{GRA}, Lemma \ref{XLT} and the fact (see \eqref{bound-B}) that $s^\alpha\gtrsim B(s)$ for any $s\geq0$ since $B\in\mathbf{B}_\aa$, we have
\begin{equation}\begin{split}\label{1W1U}
\E[\|\nn u_\vv\|_{L^2(\T^d)}^2]
\asymp \ff 1 t\sum_{\xi\in\mathbb{Z}^d\backslash\{0\}}\ff{\exp(-\vv|\xi|^2)}{|\xi|^{2}B(2\pi^2|\xi|^2)}
\gtrsim
\ff 1 t\|\phi_{c\varepsilon}\|_{l^{2(1+\aa)}(\mathbb{Z}^d)}^{2(1+\aa)},\quad t>1,\,\varepsilon>0,
\end{split}\end{equation}
where $c=1/(2+2\alpha)$. 
By Lemma \ref{g} and \eqref{1W1U}, we have
\begin{equation}\begin{split}\label{2W1U}
\E[\|\nn u_\vv\|_{L^2(\T^d)}]\gtrsim
\begin{cases}
\ff 1 {\sqrt{t}},\quad &d<2(1+\aa),\\
\ff 1 {\sqrt{t}}|\log\vv|^{\ff 1 2},\quad &d=2(1+\aa),\\
\ff 1 {\sqrt{t}}\vv^{-\ff 1 4[d-2(1+\aa)]},\quad &d>2(1+\aa),
\end{cases}
\end{split}\end{equation}
for large enough $t>0$.

Letting $\vv$ be chosen as in  \eqref{vv}, for large enough $t>0$, on the one hand, we immediately obtain that $\E[\|\nn u_\vv\|_{L^2(\T^d)}^2]\gtrsim\vv$ by \eqref{2W1U},  and on the other hand, by Proposition \ref{NUP}, we have that $\E[\|\nn u_\vv\|_{L^4(\T^d)}^4]\lesssim \vv^2$. According to \eqref{W1L}, we have
\begin{align*}
\E[\W_1(\mu_t^B,\mathfrak{m})]
&\gtrsim\ff 1 {\kappa}\vv-\ff C {\kappa^3} \vv^2,\quad \kappa>0.
\end{align*}
Letting $\kappa=2 \sqrt{C\vv}$ and $\vv$ be chosen as in \eqref{vv}, we prove \eqref{MW1L}.
\end{proof}

Now, we use Proposition \ref{sBM-main} to prove the first result in Theorem \ref{TH1}.
\begin{proof}[Proof of Theorem \ref{TH1}(1)]
It is well known that $\W_1(\mu_t^B,\mathfrak{m})\le\W_p(\mu_t^B,\mathfrak{m})$ for every $p\geq1$, since $p\mapsto \W_p$ is increasing. By the  assumption that $B\in\mathbf{B}_\aa\cap\mathbf{B}^\aa$ for some $\aa\in[0,1]$, due to \eqref{MW1L}, it is clear that \eqref{1TH1} holds.
\end{proof}

\subsection{Proofs of Theorem \ref{TH1}: the sfBM case}
In this part, we adopt another approach to establish the lower bound on $\W_p$ for all $p>0$, which is motivated by \cite[Theorem 1.1(2)]{WangWu}. One key step is to use the empirical measure associated with the time-discretized process to approximate the one associated with the original continuous-time process.

\begin{proof}[Proof of Theorem \ref{TH1}(2)]
Let $B\in\mathbf{B}_\aa$ for some $\aa\in(0,1]$ and $p\in(0,\alpha)$. For any  $N\in\N$  and $t> N$, we define
$$\mu_N^{B,H}=\ff 1 N\sum_{i=1}^N\dd_{X_{t_i}^{B,H}},$$
where $t_i=\ff{(i-1)t} N$, $1\le i\le N$. By the triangle inequality of Wasserstein distance, we have
\begin{equation}\label{THL2}
\E[\W_p(\mu_t^{B,H},\mathfrak{m})]\ge \E[\W_p(\mathfrak{m},\mu_{N}^{B,H})]-\E[\W_p(\mu_t^{B,H},\mu_N^{B,H})]
\end{equation}

Firstly, we give the upper bound of $\E[\W_p(\mu_t^{B,H},\mu_N^{B,H})]$. It is easy to verify that
$$\frac{1}{t}\sum_{i=1}^N\int_{t_i}^{t_{i+1}}\delta_{X^{B,H}_s}( \d x)\delta_{X^{B,H}_{t_i}}( \d y)\,\d s\in\mathscr{C}(\mu_t^{B,H},\mu_N^{B,H}),$$
and hence,
$$\W_p(\mu_t^{B,H},\mu_N^{B,H})\le \frac{1}{t}\sum_{i=1}^N\int_{t_i}^{t_{i+1}}\rho(X^{B,H}_{s},X^{B,H}_{t_i})^p\,\d s.$$
From \cite[page 17]{WangWu}, we see that
$$\E[(S_r^B)^p]\lesssim r^{\ff p {\aa}},\quad r\in[0,1],\,p\in(0,\aa).$$
Let $\mathrm{Proj}:\R^d\to \R^d/ \mathbb{Z}^d=\T^d$ be the natural projection map. Then
$$\rho\big(\mathrm{Proj}(x), \mathrm{Proj}(y)\big)\le|x-y|,\quad x,y\in\R^d.$$
Since $(X_t^{B,H})_{t\geq0}$ is the subordinated process of a natural projection of a $\R^d$-valued fBM,
we have
\begin{eqnarray*}\begin{split}
\E[\rho(X_s^{B,H},X_{t_i}^{B,H})^p]&=\E[\rho(X_{S_s^B}^H,X_{S_{t_i}^B}^H)^p]
\lesssim \E[|S_s^B-S_{t_i}^B|^{pH}]\\
&=\E[|S_{s-t_i}^B|^{pH}]\lesssim (s-t_i)^{\ff{pH}\aa},\quad s\geq t_i.
\end{split}\end{eqnarray*}
Thus
\begin{equation}\label{THL3}
\E[\W_p(\mu_t^{B,H},\mu_N^{B,H})]\lesssim (tN^{-1})^{\ff {pH}\aa},\quad t> N,N\in\N.
\end{equation}

Secondly, since $\T^d$ is compact, it is clear that
$$\mathfrak{m}(\{\rho(x,\cdot)^p\le r\})\lesssim r^{\ff d p},\quad r>0,\,x\in\T^d.$$
By \cite[Proposition~4.2]{K}, this implies that
\begin{equation}\label{LL}
\W_p(\mu_N^{B,H},\mathfrak{m})\gtrsim N^{-\frac{p}{d}},\quad N\in\mathbb{N}.
\end{equation}

Finally, combining \eqref{LL} with \eqref{THL2} and \eqref{THL3}, we arrive at
$$\E[\W_p(\mu_t^{B,H},\mathfrak{m})]\gtrsim N^{-\ff p d}-(tN^{-1})^{\ff {pH}\aa},\quad t>N,\,N\in\N.$$
Since $d>\aa/H$, by taking $N\asymp t^{\ff d{d-\aa/H}}$, we have for large enough $t>0$,
$$\E[\W_p(\mu_t^{B,H},\mathfrak{m})]\gtrsim t^{-\ff p{d-\aa/H}}.$$

Therefore, due to the fact that $(\W_p)^{1/p}\leq\W_1\leq\W_q$ for every $0<p\leq1\leq q$ by the H\"{o}lder's inequality,
we finish the proof of \eqref{THL1}.
\end{proof}

\subsection{Lower bounds for sfBMs on $\R^d$}
Concerning two independent sfBMs on $\R^d$, we have the following corollaries on the lower bound of $\W_p$ for all $p\geq1$. We introduce some notions first.  Let $M_1$ and $M_2$ be metric spaces and $F: M_1\rightarrow M_2$ be a Borel measurable map. Given a Borel measure $\nu$ on
$M_1$, we define the push-forward measure $F_\ast\nu$ by $F_\ast\nu(A)=\nu\big(F^{-1}(A)\big)$ for any Borel subset $A$ of $M_2$. For any $p\geq1$ and Borel probability measures $\mu,\nu$ on $\R^d$, let
$$\W_{p,\R^d}(\mu,\nu)=\inf_{\pi\in\C_{\R^d}(\mu,\nu)}\Big(\int_{\R^d\times \R^d}|x-y|^p\,\pi(\d x,\d y)\Big)^{1/p},$$
where $\C_{\R^d}(\mu,\nu)$ is the set of all probability measure on the product space $\R^d\times \R^d$ with marginal distributions $\mu$ and $\nu$, respectively.
\begin{cor}\label{WPLT}
Let $B_1\in\mathbf{B}_{\aa_1}$ and $B_2\in\mathbf{B}^{\aa_2}$ for some $\aa_1,\aa_2\in[0,1]$. Assume that $\alpha_1\leq \alpha_2$ and $X^1:=(X_t^{B_1})_{t\ge 0}$ and $X^2:=(X_t^{B_2})_{t\ge 0}$ are $\R^d$-valued, independent sBMs corresponding to $B_1$ and $B_2$, respectively. Then for every $p\ge 1$ and every large enough $t>0$,
\begin{equation}\label{WPL}
\E\Big[\W_{p,\R^d}^p\Big(\ff 1 t\int_0^t \dd_{X_s^{B_1}}\,\d s, \ff 1 t\int_0^t \dd_{X_s^{B_2}}\,\d s\Big)\Big]\gtrsim
\begin{cases}
 t^{-p/2},&\quad d<2(1+\aa_1),\\
\left(\frac{\log t}{t}\right)^{p/2},&\quad d=2(1+\aa_1),\\
t^{-p/(d-2\aa_1)},&\quad d>2(1+\aa_1).
\end{cases}
\end{equation}
\end{cor}

\begin{proof}
We use $\E_1[\cdot]$, $\E_2[\cdot]$ to denote the expectation with respect to $X^1$ and $X^2$ respectively, so that $\E[\cdot]=\E_1[\E_2[\cdot]]$. By the convexity of $\W_p$ (see e.g. \cite[Theorem 4.8]{Villani2008}), we have
\begin{equation}\begin{split}\label{1WPLT}
\E_2\Big[\W_{p,\R^d}^p\Big(\ff 1 t\int_0^t\dd_{X_s^{B_1}}\,\d s,\ff 1 t\int_0^t\dd_{X_s^{B_2}}\,\d s\Big)\Big]
&\ge \W_{p,\R^d}^p\Big(\ff 1 t\int_0^t\dd_{X_s^{B_1}}\,\d s,\ff 1 t\int_0^t \E_2[\dd_{X_s^{B_2}}]\,\d s\Big)\\
&=\W_{p,\R^d}^p\Big(\ff 1 t\int_0^t\dd_{X_s^{B_1}}\,\d s,\ff 1 t\int_0^t \mathfrak{n}_s^{B_2}\,\d s\Big),\quad p\geq1,
\end{split}\end{equation}
where for every $s>0$, $\mathfrak{n}_s^{B_2}$ is a Borel probability measure on $\R^d$ defined by
$$\mathfrak{n}_s^{B_2}(A)=\int_A\int_0^\infty \frac{1}{(2\pi u)^{d/2}}\exp\Big(-\frac{|x|^2}{2u}\Big)\,\P(S_s^{B_2}\in\d u)\d x,$$
for any Borel subset $A$ of $\R^d$.

As before, let $\mathrm{Proj}:\R^d\to\T^d$ be the natural projection map.
It is clear that $\mathrm{Proj}$ is Lipschitz continuous with Lipschitz constant $1$. Then, by the definition, it is easy to see that
\begin{equation}\label{2WPLT}
\W_{p}^p\Big(\ff 1 t\int_0^t\dd_{\mathrm{Proj}(X_s^{B_1})}\,\d s,\ff 1 t\int_0^t \mathrm{Proj}_\ast\mathfrak{n}_s^{B_2}\,\d s\Big)\le\W_{p,\R^d}^p\Big(\ff 1 t\int_0^t\dd_{X_s^{B_1}}\,\d s,\ff 1 t\int_0^t\mathfrak{n}_s^{B_2}\,\d s\Big),\quad p\geq1.
\end{equation}
For every $p\geq1$, by the triangle inequality and the elementary inequality $(a+b)^p\leq 2^{p-1}(a^p+b^p)$ for every $a,b\geq0$, we have
\begin{equation}\begin{split}\label{3WPLT}
&\W_{p}^p\Big(\ff 1 t\int_0^t\dd_{\mathrm{Proj}(X_s^{B_1})}\,\d s,\mathfrak{m}\Big)-\W_{p}^p\Big(\ff 1 t\int_0^t \mathrm{Proj}_\ast\mathfrak{n}_s^{B_2}\,\d s,\mathfrak{m}\Big)\\
&\lesssim\W_{p}^p\Big(\ff 1 t\int_0^t\dd_{\mathrm{Proj}(X_s^{B_1})}\,\d s,\ff 1 t\int_0^t \mathrm{Proj}_\ast\mathfrak{n}_s^{B_2}\,\d s\Big).
\end{split}\end{equation}

Let $t>0$ and set $\mu_t:=\ff 1 t\int_0^t \mathrm{Proj}_\ast\mathfrak{n}_s^{B_2}\,\d s$. For every $\vv>0$, let $u_\vv$ be the solution to the Poisson's equation on $\T^d$:
$$-\Delta u=P_\vv(\mu_t-\mathfrak{m}).$$
Then, applying \eqref{FM} with \eqref{GRA} , we have
\begin{equation}\label{6WPLT}
\widehat{\nn u}_\vv(\xi)=\ff{\i\xi}{2\pi|\xi|^2}\widehat{P_\vv(\mu_t-\mathfrak{m})}(\xi)
=\ff{\i\xi}{2\pi|\xi|^2}\exp(-2\pi^2\vv|\xi|^2)\widehat{\mu_t-\mathfrak{m}}(\xi),\quad \xi\in\Z^d\backslash\{0\}.
\end{equation}
Applying \eqref{LT}, since $B_2(t)\gtrsim \min\{t^{\aa_2},t\}$ for every $t\geq0$ by \eqref{bound-B}, it is easy to see that
\begin{equation}\begin{split}\label{4WPLT}
\hat{\mu}_t(\xi)&=\frac{1}{t}\int_0^t\int_{\T^d}\exp(-2\pi\i\langle\xi,x\rangle)\,\mathrm{Proj}_\ast\mathfrak{n}_s^{B_2}(\d x)\d s\\
&=\frac{1}{t}\int_0^t\int_0^\infty\exp(-2\pi^2 u|\xi|^2)\,\P(S_s^{B_2}\in\d u)\d s\\
&\lesssim\ff 1 t \int_0^t\exp\left(-(2\pi^2|\xi|^2)^{\aa_2} s\right)\,\d s\\
&\lesssim\ff 1 {t|\xi|^{2\aa_2}},\quad \xi\in\Z^d\backslash\{0\}.
\end{split}\end{equation}
Putting \eqref{4WPLT} and \eqref{6WPLT} together, we obtain
\begin{equation}\begin{split}\label{7WPLT}
|\widehat{\nn u}_\vv|(\xi)&\lesssim\ff{\exp(-2\pi^2\vv|\xi|^2)}{|\xi|}\ff{1}{t|\xi|^{2\aa_2}}
\lesssim\ff{\exp(-2\pi^2\vv|\xi|^2)}{t(1+|\xi|)^{1+2\aa_2}}=\frac{1}{t}\phi^{1+2\alpha_2}_{c\varepsilon}(\xi),\quad \xi\in\mathbb{Z}^d\setminus\{0\},
\end{split}\end{equation}
where $c={2\pi^2}/(1+2\alpha_2)$.  Hence, by the Hausdorff--Young inequality \eqref{HY} and \eqref{7WPLT}, we have
\begin{equation}\begin{split}\label{moment-grad}
\|\nn u_\vv\|_{L^p(\T^d)}^p\le\Big(\sum_{\xi\in\Z^d}|\widehat{\nn u}_\vv|^q(\xi)\Big)^{p/q}
\lesssim\ff 1 {t^p}\|\phi_{c\varepsilon}\|_{L^{q(1+2\aa_2)}(\Z^d)}^{p(1+2\aa_2)},\quad p\geq2,
\end{split}\end{equation}
where $q=p/(p-1)$.

Let $p\geq2$. Combining \eqref{moment-grad} with Lemma \ref{g} and \eqref{WPU}, we obtain
\begin{align*}
\W_{p}^p\Big(\ff 1 t\int_0^t \mathrm{Proj}_\ast\mathfrak{n}_s^{B_2}\,\d s,\mathfrak{m}\Big)
\lesssim
\begin{cases}
\inf_{\varepsilon>0}\{ \vv^{\ff p 2}+\ff 1{t^p}\},\quad &d<q(1+2\aa_2),\\
\inf_{\varepsilon>0}\{ \vv^{\ff p 2}+\ff 1{t^p}|\log \vv|^{p-1}\},\quad &d=q(1+2\aa_2),\\
\inf_{\varepsilon>0}\{ \vv^{\ff p 2}+\ff 1{t^p}\vv^{-\ff 1 2(d/q-1-2\aa_2)p}\},\quad & d>q(1+2\aa_2).
\end{cases}
\end{align*}
If $d<q(1+2\aa_2)$, then by letting $\vv=\ff 1 {t^2}$, we have
$$\W_{p}^p\Big(\ff 1 t\int_0^t \mathrm{Proj}_\ast\mathfrak{n}_s^{B_2}\,\d s,\mathfrak{m}\Big)\lesssim \ff 1 {t^p}+\ff 1 {t^p}\lesssim \ff 1 {t^p}.$$
If $d=q(1+2\aa_2)$, by letting $\vv=t^{-\ff 2 {d/q-2\aa_2}}$, we can easily verify that
\begin{align*}
\W_{p}^p\Big(\ff 1 t\int_0^t \mathrm{Proj}_\ast\mathfrak{n}_s^{B_2}\,\d s,\mathfrak{m}\Big)
&\lesssim  t^{-\ff p {d/q-2\aa_2}}+t^{-p}\Big|\log t^{-\ff 2 {d/q-2\aa_2}}\Big|^{p-1}
\lesssim t^{-p}(\log t)^{p-1}.
\end{align*}
If $d>q(1+2\aa_2)$, by letting $\vv=t^{-\ff 2 {d/q-2\aa_2}}$, then
\begin{align*}
\W_{p}^p\Big(\ff 1 t\int_0^t \mathrm{Proj}_\ast\mathfrak{n}_s^{B_2}\,\d s,\mathfrak{m}\Big)&\lesssim t^{-\ff p {d/q-2\aa_2}}+t^{-p}t^{\ff{p}{d/q-2\aa_2}(d/q-1-2\aa_2)}=2t^{-\ff p {d/q-2\aa_2}}.
\end{align*}
Gathering the above estimates together, we arrive at
\begin{equation}\begin{split}\label{8WPLT}
\W_{p}^p\Big(\ff 1 t\int_0^t \mathrm{Proj}_\ast\mathfrak{n}_s^{B_2}\,\d s,\mathfrak{m}\Big)&\lesssim
\begin{cases}
t^{-p},\quad &d<q(1+2\aa_2),\\
\big(\frac{\log t}{t}\big)^{p},\quad &d=q(1+2\aa_2),\\
t^{-\ff p {d/q-2\aa_2}},\quad &d>q(1+2\aa_2),
\end{cases}
\end{split}\end{equation}
for large enough $t>0$. Let $q=\min\{2,p/(p-1)\}$. For every $1\leq p<2$,  we also have \eqref{8WPLT} due to that $\W_p\leq\W_2$.

Let $p\geq1$. Putting \eqref{1WPLT}, \eqref{2WPLT} and \eqref{3WPLT} together and then taking expectation w.r.t. $\E_1[\cdot]$, by the fact that $\W_1\leq\W_p$, we have
\begin{equation*}\begin{split}
\E\Big[\W_{p,\R^d}^p\Big(\ff 1 t \int_0^t \dd_{X_s^{B_1}}\d s,\ff 1 t\int_0^t \dd_{X_s^{B_2}}\,\d s\Big)\Big]&\gtrsim
\E_1\Big[\W_{1}^p\Big(\frac{1}{t}\int_0^t\delta_{\mathrm{Proj}(X_s^{B_1})}\,\d s,\mathfrak{m}\Big)\Big]\\
&-\W_{p}^p\Big(\ff 1 t\int_0^t \mathrm{Proj}_\ast\mathfrak{n}_s^{B_2}\,\d s,\mathfrak{m}\Big).
\end{split}\end{equation*}
Therefore, employing \eqref{MW1L} and \eqref{8WPLT} and noting that the estimate on $\W_{1}^p\big(\ff 1 t\int_0^t \mathrm{Proj}_\ast\mathfrak{n}_s^{B_2}\,\d s,\mathfrak{m}\big)$ is dominated by the one on $\E_1\big[\W_{p}^p\big(\frac{1}{t}\int_0^t\delta_{\mathrm{Proj}(X_s^{B_1})}\,\d s,\mathfrak{m}\big)\big]$ for large enough $t>0$,  we have
\begin{equation*}
\E\Big[\W_{p,\R^d}^p\Big(\ff 1 t \int_0^t \dd_{X_s^{B_1}}\,\d s,\ff 1 t\int_0^t \dd_{X_s^{B_2}}\,\d s\Big)\Big]\gtrsim
\begin{cases}
 t^{-p/2},&\quad d<2(1+\aa_1),\\
 (\log t/t)^{p/2},&\quad d=2(1+\aa_1),\\
 t^{-\ff p {d-2\aa_1}},&\quad d>2(1+\aa_1),
\end{cases}
\end{equation*}
for large enough $t>0$, which finishes the proof.
\end{proof}

With \eqref{THL1} and Lemma \ref{SDU} in hand, by a similar argument as above, we can derive the following result. The details of proof are left to the interested reader.
\begin{rem}
Assume that $\aa_1\in[0,1]$, $\aa_2\in(0,1)$, and $B_1\in\mathbf{B}_{\aa_1}$, $B_2\in\mathbf{B}^{\aa_2}$.
For each $k=1,2$,  let $H_k\in(0,1)$, and $X^{B_k, H_k}:=(X_t^{B_k, H_k})_{t\ge 0}$ be a $\R^d$-valued sfBM with index $H_k$ corresponding to $B_k$. Suppose that $\aa_1/H_1\le \aa_2/H_2$, $X^{B_1, H_1}$ and $X^{B_2, H_2}$ are independent, and the L\'{e}vy measure corresponding to $B_2$ satisfying that $\nu(\d y)\ge cy^{-1-\aa_2}\,\d y$ for some constant $c>0$.
If $d>2+\aa_1/H_1$, then for any $p\ge 1$ and any large enough $t>0$,
$$\E\Big[\W_{p,\R^d}^p\Big(\ff 1 t\int_0^t \dd_{X_s^{B_1,H_1}}\,\d s, \ff 1 t\int_0^t \dd_{X_s^{B_2,H_2}}\,\d s\Big)\Big]\gtrsim t^{-\ff p{d-\aa_1/H_1}}.$$
\end{rem}

\section{The time-discretized sfBM case: proofs of Theorem \ref{W1TU}}
We begin with a result which extends Lemma \ref{XLT} to the present setting of time-discretized sBMs.
\begin{lem}\label{HMS}
Assume that $B\in\mathbf{B}$. Then for any $0<\tau\le t$ and $\xi\in\Z^d$,
$$ \E[|\widehat{\mu_{\tau,t}^B}(\xi)|^2]\lesssim  \ff 1 t\Big(\ff 1 {B(2\pi^2|\xi|^2)}+\tau\Big).$$
\end{lem}

\begin{proof}
Recall that $\mu_{\tau,t}^B=\mu_{\tau,t}^{B,1/2}=\frac{1}{\lfloor t/\tau\rfloor}\sum_{k=1}^{\lfloor t/\tau\rfloor}\delta_{X_{k\tau}^B}$.
Let $n=\lfloor t/\tau\rfloor$. Since $X^B$ has independent increments and $X_t^B-X_s^B$ and $X^B_{t-s}$ have the same distribution for every $t\geq s>0$, by the independence of $X^B$ and $S^B$ and \eqref{LT}, we have
\begin{eqnarray*}
\E[|\widehat{\mu_{\tau,t}^B}(\xi)|^2]
&=&\ff 1 {n^2}\sum_{l=1}^n\sum_{k=1}^n\E\left[\exp(-2\pi^2|\xi|^2S_{|k\tau-l\tau|}^B)\right]
=\ff 1 {n^2}\sum_{l=1}^n\sum_{k=1}^n\exp\left[-B(2\pi^2|\xi|^2)|k-l|\tau\right]\\
&\le&\ff 2 n\Big(1+\sum_{l=1}^{\infty}\exp\Big[-B(2\pi^2|\xi|^2) l\tau\Big]\Big)
\lesssim \ff 1 n\Big(1+\ff 1 {B(2\pi|\xi|^2)\tau}\Big),\quad \xi\in\Z^d,
\end{eqnarray*}
which completes the proof.
\end{proof}

As for the sfBM, by  Lemma \ref{SDU}, an analogous argument as \ref{HMS} leads to the following upper bound estimate for $\E[|\widehat{\mu_{\tau,t}^{B,H}}(\xi)|^2]$. The proof is not difficult, but it omitted for space.
\begin{lem}\label{HMHU}
Assume that $B$ is a Bernstein function represented by \eqref{LK} such that the L\'{e}vy measure satisfies that $\nu(\d y)\ge cy^{-1-\aa}\,\d y$ for some constants $c>0$ and $\aa\in(0,1)$. Then for any $0<\tau\le t$ and any $\xi\in\Z^d$,
$$\E[|\widehat{\mu_{\tau,t}^{B,H}}(\xi)|^2]\lesssim\ff 1 t\Big(\ff 1 {|\xi|^{\aa/H}}+\tau\Big).$$
\end{lem}


Now we present the proof of Theorem \ref{W1TU}.

\begin{proof}[Proof of Theorem \ref{W1TU}]
For every $\varepsilon>0$, $0<\tau\leq t$  and $H\in(0,1)$, let $u_\vv$ be the solution to the Poisson's equation
\begin{equation*}\label{PE-H}
-\Delta u=P_\vv(\mu_{\tau,t}^{B,H}-\mathfrak{m}).
\end{equation*}
By \eqref{F2}, \eqref{FM}, \eqref{GRA}, Lemmas \ref{HMS} and \ref{HMHU} and the assumption, we have
\begin{align*}
\E[\|\nn u_{\vv}\|_{L^2(\T^d)}^2]&=\sum_{\xi\in\Z^d\backslash\{0\}}\E[|\widehat{\mu_{\tau,t}^{B,H}}(\xi)|^2]\ff{\exp(-\vv|\xi|^2)}{4\pi^2|\xi|^2}
\lesssim \ff 1 t\sum_{\xi\in\Z^d\backslash\{0\}}\ff{\exp(-\vv|\xi|^2)}{|\xi|^2}(|\xi|^{- \aa /H}+\tau)\\
&\asymp\ff 1 t\|\phi_\varepsilon\|_{l^{2+\aa/H}(\Z^d)}^{2+\aa/H}+\ff \tau t\|\phi_\varepsilon\|_{l^2(\Z^d)}^2
\asymp\ff 1 t\|\phi_\varepsilon\|_{l^{2+\aa/H}(\Z^d)}^{2+\aa/H}+\ff 1 {t^{1+\beta}}\|\phi_\varepsilon\|_{l^2(\Z^d)}^2,
\end{align*}
for small enough $\varepsilon>0$.

(\textit{a})
Let $d<2$. Applying Lemma \ref{g}, we get
$$\|\phi_\varepsilon\|_{l^{2+\aa/H}(\Z^d)}^{2+\aa/H}\asymp 1,\quad \|\phi_\varepsilon\|_{l^2(\Z^d)}^2\asymp 1.$$
Then, by the choice of $\varepsilon=t^{-1}$, one can easily check that
\begin{equation*}\label{N1}\E[\|\nn u_\vv\|_{L^2(\T^d)}]\lesssim t^{-\ff 1 2}+t^{-\ff{1+\beta}2}\lesssim t^{-\ff 1 2}.\end{equation*}

(\textit{b})
Let $d=2$. By Lemma \ref{g},
$$\|\phi_\varepsilon\|_{l^{2+\aa/H}(\Z^d)}^{2+\aa/H}\asymp 1,\quad \|\phi_\varepsilon\|_{l^2(\Z^d)}^2\asymp |\log \vv|,$$
and hence,
$$\E[\|\nn u_\vv\|_{L^2(\T^d)}]\lesssim t^{-\ff 1 2}+t^{-\ff{1+\beta}2}|\log \vv|^{\ff 1 2}.$$
By the choice of $\varepsilon=t^{-1}$, we have
\begin{equation*}\label{N2}\E[\|\nn u_\vv\|_{L^2(\T^d)}]\lesssim t^{-\ff 1 2}.\end{equation*}

(\textit{c})
Let $2<d<2+\aa/H$. Then by  Lemma \ref{g},
$$\|\phi_\varepsilon\|_{l^{2+\aa/H}(\Z^d)}^{2+\aa/H}\asymp 1,\quad \|\phi_\varepsilon\|_{l^2(\Z^d)}^2\asymp \vv^{-\ff 1 2(d-2)},$$
which implies that
\begin{equation}\label{NUU1}
\E[\|\nn u_\vv\|_{L^2(\T^d)}^2]\lesssim t^{-1}+t^{-(1+\beta)}\vv^{-\ff 1 2(d-2)}.
\end{equation}

If $2(1+\beta)\le d<2+\aa/H$ in addition, then  by the choice of $\vv=t^{-\ff{2(1+\beta)} d}$, we have $t^{-\ff{2(1+\beta)}d}\gtrsim t^{-\ff {2\beta}{d-2}}$ and $t^{-1}\lesssim t^{-(1+\beta)}\vv^{-\ff 1 2(d-2)}$. Hence, \eqref{NUU1}  implies that
\begin{equation*}\label{N3}
\E[\|\nn u_\vv\|_{L^2(\T^d)}]\lesssim t^{-\ff{1+\beta}d}.
\end{equation*}
for large enough $t>0$. If $2<d<2(1+\beta)$ in addition, then by the choice of $\vv= t^{-1} $, which implies that $t^{-1}$ is the dominant term in \eqref{NUU1},  \eqref{NUU1} lead to that
\begin{equation*}\label{N4}
\E[\|\nn u_\vv\|_{L^2(\T^d)}]\lesssim t^{-\ff 1 2}.
\end{equation*}
for large enough $t>0$.

(\textit{d})
Let $d=2+\aa/H$. According to Lemma \ref{g},
$$\|\phi_\varepsilon\|_{l^{2+\aa/H}(\Z^d)}^{2+\aa/H}\asymp|\log \vv|,\quad \|\phi_\varepsilon\|_{l^2(\Z^d)}^2\asymp \vv^{-\ff 1 2(d-2)},$$
which leads to that
\begin{equation}\label{NUU2}
\E[\|\nn u_\vv\|_{L^2(\T^d)}^2]\lesssim t^{-1}|\log \vv|+t^{-(1+\beta)}\vv^{-\ff 1 2(d-2)}.
\end{equation}

If $2\beta<\alpha/H$ in addition, then by the choice of $\vv=t^{-\ff{2(1+\beta)} d}$,  $t^{-(1+\beta)}\vv^{-\ff 1 2(d-2)}=t^{-\frac{2(1+\beta)}{d}}$ is the leading term since $2(1+\beta)/d<1$. Hence, by \eqref{NUU2} 
\begin{equation*}\label{N5}
\E[\|\nn u_\vv\|_{L^2(\T^d)}]\lesssim  t^{-\ff {1+\beta} d}.
\end{equation*}
for large enough $t>0$. If $\aa/H\le2\beta$ in addition, then by the choice of $ \vv= \ff{\log t}{t}$,  $t^{-(1+\beta)}\vv^{-\ff 1 2(d-2)}\lesssim t^{-1}|\log \vv|$, and hence, by \eqref{NUU2} 
\begin{equation*}\label{N6}
\E[\|\nn u_\vv\|_{L^2(\T^d)}]\lesssim  t^{-\ff 1 2}|\log \vv|^{\ff 1 2}\lesssim\sqrt{\ff{\log t}{t}}.
\end{equation*}

(\textit{e}) Let $d>2+\aa/H$. By Lemma \ref{g}, we have
$$\|\phi_\varepsilon\|_{l^{2+\aa/H}(\Z^d)}^{2+\aa/H}\asymp \vv^{-\ff 1 2(d-2-\aa/H)},\quad \|\phi_\varepsilon\|_{l^2(\Z^d)}^2\asymp\vv^{-\ff 1 2(d-2)}.$$
Hence
$$\E[\|\nn u_\vv\|_{L^2(\T^d)}^2]\lesssim t^{-1}\vv^{-\ff 1 2(d-2-\aa/H)}+t^{-(1+\beta)}\vv^{-\ff 1 2(d-2)}.$$

If $ \beta\le \ff{\aa}{dH-\aa}$ in addition, then $t^{-(1+\beta)}\vv^{-\ff 1 2(d-2)}$ is the dominant term by the choice of $\vv=t^{-\ff {2(1+\beta)} d}$. Thus
\begin{equation*}\label{N7}\E[\|\nn u_\vv\|_{L^2(\T^d)}]\lesssim t^{-\ff {1+\beta} d}.\end{equation*}
  If $\beta> \ff{\aa}{dH-\aa}$ in addition, then by the choice of $\vv= t^{-\ff 2 {d-\aa/H}}$, $t^{-1}\vv^{-\ff 1 2(d-2-\aa/H)}$ becomes the dominant term. Thus
\begin{equation*}\label{N8}\E[\|\nn u_\vv\|_{L^2(\T^d)}]\lesssim t^{-\ff 1 {d-\aa/H}}.\end{equation*}

Therefore, gathering (\textit{a})--(\textit{e})  and \eqref {WPU} together, we complete the proof.
\end{proof}

\begin{rem}
By adopting the similar argument as for Proposition \ref{NUP}, the upper bound estimate on $\E[\|\nn u_\vv\|_{L^p(\T^d)}]$ for every $p\in\mathbb{N}$ even should be obtained, where for each $\vv>0$, $u_\vv$ is the solution to \eqref{PE} with $B\in\mathbf{B}^\aa$ for some $\aa\in[0,1]$. Combining this with \eqref{WPU}, one may obtain the upper bound on $\E[\W_p^p(\mu_{\tau,t}^B,\mathfrak{m})]$ for any $p> 2$. However, the proof seems rather long and complicated.
\end{rem}

\subsection*{Acknowledgment}
The authors would like to thank Prof. Feng-Yu Wang for helpful comments.

\appendix

\section*{Appendix}

\setcounter{section}{0}
\renewcommand{\thesection}{A.\arabic{section}}
\setcounter{thm}{0}
\renewcommand{\thethm}{A.\arabic{thm}}

In this part, we prove the following result employed in the proof of Lemma \ref{SDU}.
\begin{lem}\label{C}
Let $\dd>1$, $\aa\in(0,1)$, and
$$g(x)=x\Big[(1-\log x)^{\ff 1 {\dd}}-(-\log x)^{\ff 1 {\dd}}\Big]^{\dd-\aa},\quad 0<x\le 1.$$
Then the function $g$ is convex and strictly increasing on $(0,1]$.
\end{lem}

\begin{proof} It suffices to prove the assertion on $(0,1)$.

Let $\kappa\in(0,1)$, $\beta>0$, and set
$$f(z):=[(1+z)^\kappa-z^\kappa]^\beta,\quad z>0.$$
Then, it is clear that $f(z)>0$, $z> 0$, and
\begin{equation}\label{f'}
f'(z)=\beta\kappa\ff{(1+z)^{\kappa-1}- z^{\kappa-1}}{(1+z)^\kappa-z^\kappa} f(z)<0,\quad z>0.
\tag{A1}\end{equation}
Then
\begin{equation}\begin{split}\label{f''}
f''(z)&=\beta\kappa\ff{(1+z)^{\kappa-1}- z^{\kappa-1}}{(1+z)^\kappa-z^\kappa} f'(z)\\
&+\beta\kappa\ff{(\kappa-1)[(1+z)^{\kappa-2}-z^{\kappa-2}][(1+z)^\kappa-z^\kappa]-\kappa[(1+z)^{\kappa-1}-z^{\kappa-1}]^2}{[(1+z)^\kappa-z^\kappa]^2}f(z)\\
&=\beta^2\kappa^2\Big[\ff{(1+z)^{\kappa-1}- z^{\kappa-1}}{(1+z)^\kappa-z^\kappa}\Big]^2 f(z)\\
&+\beta\kappa\ff{(\kappa-1)[(1+z)^{\kappa-2}-z^{\kappa-2}][(1+z)^\kappa-z^\kappa]
-\kappa[(1+z)^{\kappa-1}-z^{\kappa-1}]^2}{[(1+z)^\kappa-z^\kappa]^2}f(z),\quad z>0.
\end{split}\tag{A2}\end{equation}
We \textbf{claim} that $f''(z)\ge f'(z)$ for all $z> 0$.

Since $\kappa\in(0,1)$, it is obviously that
$$(1+z)^{\kappa-1}-z^{\kappa-1}\le0,\quad z> 0.$$
This together with $f(z)>0$ derive that
\begin{equation}\begin{split}\label{BK}\beta^2\kappa^2\Big[\ff{(1+z)^{\kappa-1}- z^{\kappa-1}}{(1+z)^\kappa-z^\kappa}\Big]^2 f(z)\ge\beta\kappa\ff{(1+z)^{\kappa-1}- z^{\kappa-1}}{(1+z)^\kappa-z^\kappa} f(z),\quad z> 0.\end{split}\tag{A3}\end{equation}
Let
$$l(z):=(\kappa-1)[(1+z)^{\kappa-2}-z^{\kappa-2}][(1+z)^\kappa-z^\kappa]-\kappa[(1+z)^{\kappa-1}-z^{\kappa-1}]^2,\quad z> 0.$$
Next, we prove $l(z)\ge 0$ for $z\ge 0$.
We rewrite $l(z)$ as
\begin{align*}
l(z)&=(\kappa-1)z^{\kappa-2}\Big[\Big(1+\ff 1 z\Big)^{\kappa-2}-1\Big]z^\kappa\Big[\Big(1+\ff 1 z\Big)^\kappa-1\Big]-\kappa z^{2(\kappa-1)}\Big[\Big(1+\ff 1 z\Big)^{\kappa-1}-1\Big]^2\\
&=z^{2(\kappa-1)}\Big\{(\kappa-1)\Big[\Big(1+\ff 1 z\Big)^{\kappa-2}-1\Big]\Big[\Big(1+\ff 1 z\Big)^\kappa-1\Big]-\kappa\Big[\Big(1+\ff1 z\Big)^{\kappa-1}-1\Big]^2\Big\}.
\end{align*}
Let $a:=1+1/z>1$ and set
$$l_1(a):=(\kappa-1)[a^{\kappa-2}-1][a^\kappa-1]-\kappa[a^{\kappa-1}-1]^2,\quad a\ge 1.$$
Then
\begin{align*}
l_1(a)&=(\kappa-1)[a^{2(\kappa-1)}-a^{\kappa-2}-a^\kappa+1]-\kappa[a^{2(\kappa-1)}-2a^{\kappa-1}+1]\\
&=-a^{2(\kappa-1)}-(\kappa-1)a^{\kappa-2}-(\kappa-1)a^\kappa+\kappa-1+2\kappa a^{\kappa-1}-\kappa\\
&=-a^{2(\kappa-1)}-(\kappa-1)a^{\kappa}+2\kappa a^{\kappa-1}-(\kappa-1)a^{\kappa-2}-1.
\end{align*}
Furthermore,
\begin{align*}
&l_1^{'}(a)=-2(\kappa-1)a^{2\kappa-3}-\kappa(\kappa-1)a^{\kappa-1}+2\kappa(\kappa-1)a^{\kappa-2}-(\kappa-1)(\kappa-2)a^{\kappa-3}\\
&=(\kappa-1)a^{\kappa-3}[-2a^\kappa-\kappa a^2+2\kappa a-\kappa+2]
=(\kappa-1)a^{\kappa-3}[-\kappa(a-1)^2-2(a^\kappa-1)]\ge 0.
\end{align*}
This implies that $l_1(a)$ is an increasing function on $[1,\infty)$. Then for any $a\ge 1$, we have $l_1(a)\ge l_1(1)=0$, $a\ge 1$.
Thus, $l(z)=z^{2(\kappa-1)}l_1(1+1/z)\ge 0$. This together with \eqref{f'}, \eqref{f''} and \eqref{BK} imply that
\begin{equation}\label{f'''}
f''(z)\ge f'(z),\quad z> 0.
\tag{A4}\end{equation}

Let $z:=-\log x$, $0<x<1$, $\kappa:=1/\dd\in(0,1)$, $\beta:=\dd-\aa>0$.
Set
$$h(z):=\e^{-z}[(1+z)^\kappa-z^\kappa]^\beta,\quad z> 0.$$
Note that $h(0)=1$, $h(z)>0$, $z> 0$, and  $h(z)=\e^{-z}f(z)$. Then
\begin{equation}\label{h'}
h'(z)=\e^{-z}f'(z)-\e^{-z}f(z)< 0,\quad z>0.\tag{A5}
\end{equation}
Hence,
$$g'(x)=-\ff{h'(-\log x)} x> 0,\quad x\in(0,1),$$
which implies that $g$ is strictly increasing on $(0,1)$.

By \eqref{f'''} and \eqref{h'}, we obtain that
\begin{align*}
h''(z)&=-\e^{-z}f'(z)+\e^{-z}f''(z)+\e^{-z}f(z)-\e^{-z}f'(z)\\
&=\e^{-z}[f''(z)-2f'(z)+f(z)]\ge -h'(z),\quad z>0.
\end{align*}
Thus,
$$g''(x)=\ff 1 {x^2}\big[h''(-\log x)+h'(-\log x)\big]\ge 0,\quad x\in(0,1),$$
which implies that $g$ is convex on $(0,1)$.
\end{proof}

\end{document}